\documentclass[11pt, reqno]{amsart}

\usepackage{amssymb,amsmath,amsthm,enumerate,latexsym, graphics,shapepar,nccmath,mathtools}
\usepackage{tikz-cd}
\usepackage{libertinus}
\usepackage[utf8]{inputenc}
\usepackage[pagebackref,colorlinks=true, pdfstartview=FitV, linkcolor=red, citecolor=blue, urlcolor=blue]{hyperref}
\usepackage{cleveref}
\usepackage[margin=0.9in]{geometry}

\usepackage{parskip}
\usepackage{slashbox}
\newtheorem{thm}{Theorem}[section]
\newtheorem{lemma}[thm]{Lemma}
\newtheorem{cor}[thm]{Corollary}
\newtheorem{prop}[thm]{Proposition}
\theoremstyle{definition}

\newtheorem{conjecture}[thm]{Conjecture}

\newtheorem{defn}[thm]{Definition}
\newtheorem{example}[thm]{Example}
\newtheorem{ex-notn}[thm]{Example/Notation}

\newtheorem{question}[thm]{Question}
\newtheorem{rmk}[thm]{Remark}

\newcommand{\sk}{{\ensuremath{\sf k }}}
\newcommand{\m}{\ensuremath{\mathfrak m}}
\newcommand{\BN}{\ensuremath{\mathbb{N}}}
\newcommand{\BZ}{\ensuremath{\mathbb{Z}}}
\newcommand{\BQ}{\ensuremath{\mathbb{Q}}}
\newcommand{\BB}{\ensuremath{\mathbb{B}}}
\newcommand{\mP}{\ensuremath{\mathcal{P}}}
\newcommand{\G}{\ensuremath{\mathbb{G_\bullet}}}
\newcommand{\F}{\ensuremath{\mathbb{F_\bullet}}}

\newcommand{\lrar}{\ensuremath{\longrightarrow}}
\newcommand{\rar}{\ensuremath{\rightarrow}}
\newcommand{\ov}{\ensuremath{\overline}}

\DeclareMathOperator{\pdim}{pdim}
\DeclareMathOperator{\cmd}{cmd}
\DeclareMathOperator{\ann}{ann}
\DeclareMathOperator{\tor}{Tor}
\DeclareMathOperator{\soc}{soc}

\DeclareMathOperator{\edim}{edim}

\DeclareMathOperator{\reg}{reg}
\DeclareMathOperator{\depth}{depth}
\DeclareMathOperator{\codim}{codim}

\newcommand{\curv}{\mbox{\rm curv}}

\begin{document}
\author[H. Ananthnarayan]{H. Ananthnarayan}
\address{Department of Mathematics, I.I.T. Bombay, Powai, Mumbai 400076.}
\email{ananth@math.iitb.ac.in}

\author{Omkar Javadekar}
\address{Chennai Mathematical Institute, Siruseri, Tamilnadu 603103. India}
\email{omkarj@cmi.ac.in}

\author[Rajiv Kumar]{Rajiv Kumar}
\address{Department of Mathematics, Indian Institute of Technology Jammu, J\&K, India - 181221.}
\email{rajiv.kumar@iitjammu.ac.in}

\title{Purity of Extremal Rays of Betti Cones}

\subjclass{13A02, 13D02, 13D40}

\keywords{Betti cone, Pure resolution, Koszul algebra, Minimal multiplicity, Boij--S\"oderberg theory}
\begin{abstract}
Let $R$ be a standard graded algebra over an infinite field $\mathsf k$, and let 
$\mathbb{B}_{\mathbb{Q}}(R)$ and $\mathbb{B}_{\mathbb{Q}}^{\mathrm{pure}}(R)$ denote the rational cones spanned by the Betti tables of all finitely generated $R$-modules and of those with pure resolutions, respectively. We establish several necessary conditions for the equality $\mathbb{B}_{\mathbb{Q}}(R) = \mathbb{B}_{\mathbb{Q}}^{\mathrm{pure}}(R)$. 
When $\operatorname{edim}(R)\ge 2$, we prove that $\mathsf k$ has a pure resolution if and only if it has a linear resolution, and consequently, if the extremal rays of $\mathbb{B}_{\mathbb{Q}}(R)$ are pure, then $R$ is Koszul and good (in the sense of Roos). We show that if $R$ has depth zero, it must be Artinian for the equality of the two cones to hold. For rings with linear pairs of exact 
zerodivisors, we show that the equality of the cones implies that the $h$-polynomial has degree at most $2$, and use it to characterize generic Gorenstein Artin algebras satisfying $\mathbb{B}_{\mathbb{Q}}(R) = \mathbb{B}_{\mathbb{Q}}^{\mathrm{pure}}(R)$. 
We also characterize algebras whose extremal rays are exactly the Betti tables of shifts of $R/\mathfrak m^j$ and of pure modules $M$ with $\operatorname{codim}(M)=\operatorname{pdim}(M)$: apart from polynomial rings,  these are precisely Cohen--Macaulay algebras of dimension at most one with minimal multiplicity. In addition, we obtain a characterization of Cohen--Macaulay algebras of minimal multiplicity in terms of the extremal rays of the Betti cone of maximal Cohen--Macaulay modules.
\end{abstract}
\maketitle

\section{Introduction}
Let $R$ be a standard graded algebra over an infinite field $\mathsf{k}$. 
The central object of study in Boij--S\"oderberg theory is the rational Betti cone $\mathbb{B}_{\mathbb{Q}}(R)$, defined as the cone generated by all nonnegative rational linear combinations of Betti tables of finitely generated graded $R$-modules. Understanding elements of $\mathbb{B}_{\mathbb{Q}}(R)$ is therefore equivalent to understanding its fundamental building blocks, namely the extremal rays.  

Motivated by the \emph{multiplicity conjecture} of Herzog--Huneke--Srinivasan (see~\cite{HS98}), Boij and S\"oderberg conjectured in~\cite{BS08} that for $R = \mathsf{k}[X_1,\ldots,X_n]$, the extremal rays of $\mathbb{B}_{\mathbb{Q}}(R)$ are spanned by the Betti tables of Cohen--Macaulay modules with pure resolutions. 
They supported this conjecture by proving it in the case $n=2$, and, assuming its validity in general, derived a proof of the multiplicity conjecture.
Subsequently, Eisenbud and Schreyer proved the Boij--S\"oderberg conjectures in~\cite{ES09}. In addition, they provided an algorithm for decomposing an arbitrary Betti table in $\mathbb{B}_{\mathbb{Q}}(R)$ into a finite nonnegative rational combination of extremal Betti tables. Thus, over a polynomial ring, the Betti cone $\mathbb B_{\mathbb Q}(R)$ coincides with its  subcone $\mathbb B_{\mathbb Q}^{\rm{pure}}(R)$, generated by all pure Betti tables. 
While a complete description of the Betti cone $\mathbb{B}_{\mathbb{Q}}(R)$ is known when $R$ is a polynomial ring, very little is understood in the general setting of standard graded $\mathsf{k}$-algebras. In particular, characterizing those algebras for which the equality $\mathbb{B}_{\mathbb{Q}}(R)=\mathbb{B}_{\mathbb{Q}}^{\mathrm{pure}}(R)$ holds is an extremely difficult problem. Our article aims to establish several necessary conditions for the equality $\mathbb{B}_{\mathbb{Q}}(R)=\mathbb{B}_{\mathbb{Q}}^{\mathrm{pure}}(R)$ to hold, and provides general guiding questions and conjectures for this equality.

The equality of the pure and Betti cones has been established only in a handful of cases among standard graded $\mathsf{k}$-algebras that are not polynomial rings: for quadric hypersurfaces by Berkesch--Burke--Erman--Gibbons~\cite{BBEG12}, for short Gorenstein rings of embedding dimension two by Gibbons~\cite{Gi13}, and for Cohen--Macaulay $\mathsf{k}$-algebras of dimension at most one with minimal multiplicity by Ananthnarayan--Kumar \cite{AK20}. More recently, Iyengar--Ma--Walker~\cite{IMW23} extended the results for polynomial rings to modules of finite projective dimension over all finitely generated graded rings admitting linear Noether normalizations. In particular, they showed that if $R$ is a Cohen--Macaulay standard graded $\mathsf{k}$-algebra, then the Betti cone generated by modules of finite projective dimension coincides with the corresponding pure Betti cone. Nevertheless, a complete description of $\mathbb{B}_{\mathbb{Q}}(R)$ remains widely open in general, since for non-polynomial rings there exist modules of infinite projective dimension whose Betti tables are in $\mathbb B_{\mathbb Q}(R)$.

Due to the work of Ananthnarayan--Kumar \cite{AK20} it is known that given any standard graded algebra $R$ with the homogeneous maximal ideal $\m$, the Betti tables of shifts of the modules of the form (a) $R/\m^j$, where $j \geq 0$, and (b) the modules $M$ with pure resolution satisfying $\codim(M)=\pdim_R(M)$ always span an extremal ray in $\mathbb B_{\mathbb Q}(R)$. In fact, as noted earlier, these are the only modules that give rise to extremal Betti tables when $R$ is a polynomial ring. In a way, (a) and (b) are form the \emph{smallest} possible set of extremal rays. One of the main results of this article characterizes all standard graded $\mathsf k$-algebras $R$ for which (a) and (b) give a complete set of extremal rays of $\mathbb B_{\mathbb Q}(R)$. We show that such rings are precisely Cohen--Macaulay of dimension at most one with minimal multiplicity or are polynomial rings (see Theorem~\ref{thm:dim-at-most-1-for-smallest-set-of-extremals}). This result was established by Ananthnarayan--Javadekar--Kumar \cite{AJK26} for the class of fibre products of dimension $\geq 1$. In addition, we provide a complete description of the Betti cone of the maximal Cohen--Macaulay modules over Cohen--Macaulay algebras of any dimension with minimal multiplicity (see Proposition~\ref{MCM_Betti_Cone}), and further show that this description characterizes such algebras (see Theorem~\ref{thm:min-mult-char}). 

Although some attempts have been made to establish the equality $\mathbb{B}_{\mathbb{Q}}(R)=\mathbb{B}_{\mathbb{Q}}^{\mathrm{pure}}(R)$ for specific classes of rings, as discussed above, general necessary and sufficient conditions for this equality remain largely unknown. A major portion of this article is devoted to deriving several necessary conditions. We begin by proving a notable result that a minimal free resolution of the residue field over a standard graded $\mathsf k$-algebra $R$ with $\edim(R)\geq 2$ is pure if and only if it is linear, i.e., if and only if the algebra is Koszul (see Theorem~\ref{Koszul}). We use this to obtain the first necessary condition for the equality $\mathbb{B}_{\mathbb{Q}}(R)=\mathbb{B}_{\mathbb{Q}}^{\mathrm{pure}}(R)$, and show that $R$ must be Koszul for this equality to hold. In fact, we prove Koszulness under the weaker hypothesis that the extremal rays of $\mathbb{B}_{\mathbb{Q}}(R)$ are pure, rather than assuming the full equality above (see Theorem~\ref{thm:purityImpliesKoszul}).
Using this, we further show that if the two cones coincide, then all finitely generated $R$-modules must have rational Poincar\'e series with a common denominator, i.e., $R$ must be \emph{good} in the sense of Roos (see Theorem~\ref{thm:good}). It is worth noting that the classes of rings for which the equality $\mathbb{B}_{\mathbb{Q}}(R)=\mathbb{B}_{\mathbb{Q}}^{\mathrm{pure}}(R)$ is known to hold are all Golod over a complete intersection, a property which is stronger that being good in the sense of Roos. However, there are no known examples of good Koszul algebras that are not Golod over complete intersections. 

After proving the above necessary conditions, we focus on some special classes of $\sk$-algebras and study the equality of the pure and Betti cones for them. We prove that if the equality holds for an algebra $R$ with $\depth(R)=0$, then $R$ must be an Artinian and level algebra (see Theorem~\ref{thm:depthZeroImpliesArtinianLevel}). This result can be seen as the evidence for the conjecture that if $\mathbb{B}_{\mathbb{Q}}(R)=\mathbb{B}_{\mathbb{Q}}^{\mathrm{pure}}(R)$, then $R$ is Cohen--Macaulay level algebra. We next look at algebras with linear pair of exact zerodivisors. We show that the equality of the two cones over such rings forces the $h$-polynomial of $R$ to have degree at most $2$ (see Theorem~\ref{notSufficient}). We use this result to characterize generic Gorenstein Artin algebras satisfying $\mathbb{B}_{\mathbb{Q}}(R)=\mathbb{B}_{\mathbb{Q}}^{\mathrm{pure}}(R)$ (see Theorem~\ref{thm:generic-Gorenstein}).

This article is structured as follows. In Section \ref{sec:prelim}, we collect definitions, set-up the notation, and recall known results relevant for the rest of the article. Section \ref{sec:purity-general} is devoted to the study of purity of the resolution of $\mathsf k$ and its applications to the equality $\mathbb{B}_{\mathbb{Q}}(R)=\mathbb{B}_{\mathbb{Q}}^{\mathrm{pure}}(R)$. Section \ref{sec:purity-special} contains the analysis of this equality for the depth zero algebras and algebras with linear pairs of exact zerodivisors. It also characterizes generic Gorenstein algebras for which the two cones are equal. In Section \ref{sec:purity-and-minimal-multiplicity}, we characterize all the rings with the \emph{smallest} possible set of extremal rays by showing that they are either polynomial rings or are Cohen--Macaulay of minimal multiplicity with dimension at most one. Section \ref{special-Hilbert-series} gives a complete description of the Betti cone of maximal Cohen--Macaulay modules over Cohen--Macaulay algebras of minimal multiplicity, and provides a characterization of such algebras in terms of the description obtained. Finally, in Section \ref{sec:concluding-remarks}, we make concluding remarks and state a few open questions and conjectures. 

\section{Preliminaries}  \label{sec:prelim}

Throughout this article, we assume that $\sk$ is an infinite field, and $R$ is a standard graded Noetherian $\sk$-algebra. Unless stated otherwise, all modules considered in this article are graded and finitely generated. 

\subsection{Graded Free Resolutions}
 Let $M$ be a finitely generated graded $R$-module and 
\[\mathbb F_{\bullet} : \cdots \xrightarrow[]{} F_n \xrightarrow[]{\phi_n} F_{n-1} \xrightarrow[]{\phi_{n-1}} \cdots \xrightarrow[]{\phi_1}F_0 \xrightarrow[]{\phi_0} M \xrightarrow[]{} 0\]
be a graded minimal free resolution of $M$ over $R$, i.e., for each $i$, the map $\phi_i$ is  graded  of degree zero, and $\ker(\phi_i) \subset \m F_i$. 

\begin{enumerate}[{\rm (a)}]

\item The module $\Omega_i^R(M)=\ker(\phi_{i-1})$ is a graded $R$-module, called the \emph{$i^{th}$ syzygy module} of $M$. We define $\beta_i^R(M)= \dim_{\mathsf k}\left(\tor^R_i(M, \sk)\right)$, called the $i^{th}$ Betti number of $M$. Similarly, we define $\beta_{i,j}^R(M)= \dim_{\mathsf k}\left(\tor^R_i(M, \sk)_j\right)$, called the $(i,j)^{th}$ graded Betti number of $M$. Note that $\beta_i^R(M)$ is the number of elements in a minimal generating set of $\Omega_i^R(M)$, and $\beta_{i,j}^R(M)$ is the number of elements of degree $j$ in a minimal generating set of $\Omega_i^R(M)$. 

\item Let $\beta_{i,j}\coloneqq\beta_{i,j}^R(M)$. Then the \emph{Betti table} of $M$ is written as 

\begin{center}
\begin{tabular}{|c|c|}
\hline
\backslashbox{$j$}{$i$} & $\cdots \quad i \quad \cdots$ \\
\hline
$\vdots$ & $\vdots$ \\
$j$ & $\cdots \quad \beta_{i,\,i+j} \quad \cdots$ \\
$\vdots$ & $\vdots$ \\
\hline
\end{tabular}
\end{center}

Observe that the $(i,j)^{th} $ entry in $\beta^R(M)$ is $\beta_{i,i+j}$.
\item The series $\mathcal{P}^R_M(z)= \sum_{i\geq 0} \beta_i^R(M)z^i$ is called the \emph{Poincar\'e series of $M$}. When there is no scope for confusion, we denote $\mathcal P_M^R(z)$ by $\mathcal P_M(z)$.

\item 
The \emph{Castelnuovo--Mumford regularity of $M$} is defined as
\[ \reg_R(M) = \sup \{ j-i \mid \beta_{i,j}^R(M) \neq 0 \text{ for some } i\geq0 \text{ and }  j \in \mathbb Z\}. \]
When there is no scope for confusion, we denote $\reg_R(M)$ by $\reg(M)$. 

\item A module $M$ generated in degree $d$ is said to have a \emph{linear resolution} if $\beta_{i,j}^R(M)\neq 0$ implies $j=d+i$. In other words, all nonzero entries of $\beta^R(M)$ occur in a single row.

\item We say that $R$ is \emph{Koszul} if $\sk$ has a linear resolution over $R$.

\item A module $M$ is said to have a \emph{pure resolution} if given any $i \geq 0$,  $\beta^R_{i,j}(M) \neq 0$ for at most one $j$. In other words,  every nonzero column of $\beta^R(M)$ contains exactly one nonzero entry. In this article, by a \emph{pure module}, we mean a module having a pure resolution. 

\item Let $M$ be a pure $R$-module. We say that $M$ is of \emph{type}  
\begin{enumerate}[{\rm i)}]
\item ${{\bf d}}=(\delta_0, \delta_1,\delta_2,\ldots)$ if $\pdim_R(M)= \infty$ and $\beta_{i,\delta_i}^R(M) \neq 0$ for all $i \geq 0$. 

\item ${\bf d}= (\delta_0, \delta_1, \ldots, \delta_p, \infty, \infty,\ldots)$ if $\pdim_R(M)=p$ and $\beta_{i,\delta_i}^R(M) \neq 0$ for $0 \leq i \leq p$.
\end{enumerate}

\end{enumerate}

\begin{rmk}\label{rmk:Hilbert-series-etc}
Let $R$ be a standard graded $\sk$-algebra, and $M=\bigoplus_{i\in \mathbb Z} M_i$ be a finitely generated $R$-module.
\begin{enumerate}[{\rm (a)}]
    \item The \emph{Hilbert series of $M$} is defined as $H_M(z)=\sum_{i \in \mathbb Z} \dim_{\mathsf k}(M_i) z^i$.
    \item It is well-known that if $M\neq 0$, then there exists $f(z)\in \mathbb Z[z,z^{-1}]$ such that $H_M(z)=f(z)/(1-z)^d$, where $f(1) \in \mathbb N$ and $d = \dim(M)$ (see \cite[Section 4.1]{BH93}). When $M=R$, the number $f(1)$ is called the \emph{multiplicity} of $R$, and we denote it by $e(R)$.

    \item Let $M$ have a linear resolution. Then, by the additivity of the Hilbert series on exact sequences, we have $H_M(z)=\sum\limits_i (-1)^i H_R(z)\beta_{i,i}^R(M)z^i$. Thus, we see that $$H_M(z)=H_R(z)\sum\limits_i\beta_{i,i}^R(M)z^i=H_R(z)\mathcal P_M^R(-z).$$
\end{enumerate}
\end{rmk}

\subsection{Betti Cones and Extremal Rays} 

Let $R$ be a standard graded $\sk$-algebra.
\begin{enumerate}[a)]
\item Then the \emph{Betti cone of finitely generated $R$-modules} is defined as
\[ \mathbb{B}_{\mathbb{Q}}(R) = \{ c_1 \beta^R(M_1) + \cdots +c_n \beta^R(M_n) \mid c_i \in \mathbb{Q}_{\geq 0}, M_i {\text{\ is a finitely generated\ }} R{\text{-module}}\}. \]
Similarly, we define the \emph{Betti cone of pure $R$-modules} as
\[ \BB^{\rm{pure}}_{\mathbb{Q}}(R) = \{ c_1 \beta^R(M_1) + \cdots +c_n \beta^R(M_n) \mid c_i \in \mathbb{Q}_{\geq 0}, M_i {\text{\ is a finitely generated pure\ }} R{\text{-module}}\}. \]

\item A Betti table $\beta^R(M)$ of a nonzero finitely generated $R$-module $M$ is said to be \emph{extremal} in the Betti cone $\mathbb{B}_{\mathbb{Q}}(R)$ if whenever there exist $c_1, \ldots,c_n\in \mathbb{Q}_{\geq 0}$ and finitely generated $R$-modules $M_1,\ldots,M_n$ such that
$\beta^R(M) = \sum_{i=1}^{n} c_i \beta^R(M_i)$, then we have  $\beta^R(M)= c \beta^R(M_i)$ for some $i \in \{1,\ldots, n\}$ and $c \in \mathbb{Q}_{>0}$.

\item The condition $\BB_{\mathbb{Q}}(R)=\BB^{\rm{pure}}_{\mathbb{Q}}(R)$ is equivalent to saying that every Betti table is a finite $\mathbb Q_{\geq 0}$-linear combination of pure Betti tables. In particular, when the equality holds, we see that all extremal rays of $\mathbb B_{\mathbb Q}(R)$ are generated by Betti tables of pure modules.
\end{enumerate}

\begin{rmk}\label{bettiTableremark} Let $R$ be a standard graded $\sk$-algebra. 
\begin{enumerate}[{\rm (i)}]
    \item 
    If $c_1,\ldots, c_n \in \mathbb Z_{\geq 0}$ and $M_1, \ldots, M_n$ are finitely generated graded $R$-modules, then we have the equality $\sum\limits_{i=1}^n c_i \beta^R(M_i)= \beta^R\left(\bigoplus\limits_{i=1}^n M_i^{c_i}\right)$. Thus, every $\beta \in \mathbb B_{\mathbb Q}(R)$ can be written as $\beta= c \beta^R(M)$ for some $R$-module $M$ and $c \in \mathbb Q_{\geq 0}$.

    \item Suppose $\mathbb B_{\mathbb Q}(R)=\mathbb B _{\mathbb Q}^{\rm pure}(R)$. Then, for any $R$-module $M$, there exist pure $R$-modules $M_i$ of type ${\bf d}_i$ with ${\bf d}_i\neq {\bf d}_j$ for $i\neq j$ and $c_i\in\BQ_{\geq 0}$ such that $\beta^R(M)=\sum_ic_i\beta^R(M_i).$
    \end{enumerate}

\end{rmk}

The following result of \cite{AK20} gives a class of modules over standard graded $\sk$-algebras whose Betti table spans an extremal ray in $\BB_{\mathbb{Q}}(R)$.

\begin{thm}[\cite{AK20}, Theorem 2.4]\label{SomeExtremalRays} Let $R$ be a standard graded $\sk$-algebra. Then the Betti tables of
the shifts of the following modules span extremal rays of $\BB_{\mathbb{Q}}(R)$: 
\begin{enumerate}[{\rm (a)}]
\item A pure $R$-module $M$, with $\codim(M) = \pdim(M)$.
\item $R/\m^j$ where $j\in \mathbb{N}$.
\end{enumerate}
\end{thm}
Observe that since $\codim(M)$ is always finite, the modules of type (a) above have finite projective dimension.

\begin{rmk}\label{rem:R-mod-regular-seq-are-extremal}{\rm 
    Observe that if $\underline{x}=x_1,\ldots, x_n$ is an $R$-regular sequence, then $\cmd(R/\langle \underline{x} \rangle)=\cmd(R)$, or equivalently, $\codim(R/ \langle \underline{x}\rangle ) = \pdim(R/ \langle \underline{x}\rangle )$. Thus, if all the elements of the sequence $\underline{x}$ have the same degree, then by Theorem \ref{SomeExtremalRays}, we see that $\beta^R(R/ \langle \underline{x}\rangle )$ spans an extremal ray of $\mathbb B_{\mathbb Q}(R)$. 
}\end{rmk}

\begin{thm}[\cite{AK18}, Theorem 3.9]\label{thm:Herzog-Kuhl-in-general}
Let $R$ be a standard graded $\sk$-algebra and $M$ be a pure module of type ${{\bf d}}= (\delta_0,\ldots,\delta_p, \infty,\infty,\ldots)$. Then the following are equivalent:
\begin{enumerate}[{\rm (i)}]
\item $\codim(M)=\pdim(M)$.
\item $\beta_{i,\delta_i}(M) = \beta_{0,\delta_0}(M)  \prod\limits_{k \neq i,0} \left\vert \dfrac{\delta_k-\delta_0}{\delta_k-\delta_i} \right\vert$ for all $1\leq i \leq p$.
\end{enumerate}
\end{thm}

\begin{thm}[\cite{AK20}, Theorem 3.4, Theorem 4.3]\label{dim0and1extremals} Let $R$ be a standard graded $\sk$-algebra with $H_R(z)= (1+nz)/(1-z)^d$, $d\leq 1$. Then \begin{enumerate}[{\rm (a)}]
\item $\BB_{\mathbb{Q}}(R)=\BB^{\rm{pure}}_{\mathbb{Q}}(R)$.
\item \begin{enumerate}[{\rm (i)}]
				\item If $d=0$, then the extremal rays of $\BB_{\mathbb{Q}}(R)$ are spanned
by the Betti tables of the shifts of $R$, and $\sk$.
				\item If $d=1$, then  the extremal rays of $\BB_{\mathbb{Q}}(R)$ are spanned
by the Betti tables of the shifts of $R$, $R/\m^i$, and $R/\langle l\rangle^i$, where $i \in \mathbb{N} $ and $l$ is a linear regular element in $R$.
         \end{enumerate}
\end{enumerate} 
 In particular, $\BB_{\mathbb{Q}}(R)=\BB^{\rm{pure}}_{\mathbb{Q}}(R)$.
\end{thm}

\begin{rmk}\label{rem:directSumofR/m^j}
\cite[Lemma 2.3]{AK20}
Let $R$ be a standard graded $\mathsf k$-algebra, and $h_R(j)=\dim(R_j)$. If $M$ is an $R$-module generated in degree zero, then $\beta_{1,j}^R(M)\leq h_R(j)\beta_{0,0}^R(M)$ for all $j \in \mathbb N$. Furthermore, equality holds for some $j$ if and only if $M$ is a direct sum of copies of $R/\mathfrak m^j$ (and hence, $\beta_{1,k}^R(M)=0$ for $k \neq j$).    
\end{rmk}

\begin{rmk}\label{rem:generators-up-to-degree-j}
\cite[Lemma 4.6]{AK20} Let $M$ be a finitely generated graded $R$-module,
Let $M^{(j)}$ be the submodule of $M$ generated by elements of degree at most $j$. Then $\Omega_i^R(M)^{(j+i)} \simeq \Omega_i^R(M^{(j)})^{(j+i)}$ and $\beta_{i, k+i}^R(M) = \beta_{i, k+i}^R(M^{(j)})$ for $k \leq j$ and $i \geq 0$.
\end{rmk}

\subsection{Koszul Algebras} Recall that a $\sk$-algebra $R$ is Koszul if $\sk$ has a linear resolution over $R$. In this subsection, we collect the necessary known facts about Koszul algebras.

The following theorem due to Avramov\textendash Peeva gives a characterization of Koszul algebras.
\begin{thm}{\rm(\cite{AP01})}\label{KoszulAEP}
 Let $R$ be standard graded $\sk$-algebra. Then the following are equivalent:
 \begin{enumerate}[{\rm (i)}]
     \item $\reg(M)$ is finite for every $R$-module $M$.
     \item $\reg(\sk)$ is finite.
     \item $R$ is Koszul.
 \end{enumerate}
\end{thm}

The linearity of the resolution of $\sk$ is in fact equivalent to the weaker property of purity of the resolution. We show this in Theorem \ref{Koszul}.

The following result of Avramov--Eisenbud provides an upper bound on the Castelnuovo--Mumford regularity of an $R$-module $M$.
\begin{thm}[\cite{AE92}, Theorem 1]\label{Thm-reg-over-poly-ring}
Let $R$ be a Koszul algebra, and let $S$ be the 
polynomial ring mapping onto $R$. Then the Castelnuovo--Mumford regularity of any module $M$ over $R$ is 
finite; in fact, \[\reg_R(M) \leq \reg_S(M).\]
\end{thm}

The next theorem is a useful tool to construct Koszul algebras, and also to detect Koszulness.
\begin{thm}[\cite{CDR13}, Theorem 3.2]\label{Thm-transfer-of-Koszulness}
Let $R$ be a Koszul algebra, and let $R'$ be a quotient of $R$. Then\begin{enumerate}[{\rm (a)}]
\item If $\reg_R(R') \leq 1$ and $R$ is Koszul, then $R'$ is Koszul.

\item If $\reg_R(R')$ is finite and $R'$ is Koszul, then $R$ is Koszul.
\end{enumerate}
\end{thm}

\begin{rmk}\label{Rmk-transfer-of-Koszulness}{\rm As a consequence of the above theorem, we see that if $x\in R_1$ is a nonzerodivisor on $R$, then $R$ is Koszul if and only if $R/\langle x \rangle$ is Koszul.
}\end{rmk}

\section{Purity of \texorpdfstring{$\sk$}{k} and its applications}\label{sec:purity-general}

Fix $n\in \BN$, and let $R=\sk[X]/\langle X^n\rangle$. Then $\sk$ has a pure resolution over $R$. Moreover, the resolution of $\sk$ is linear if and only if $n \leq 2$. Note that in this case we have $\mu(\m)=1$. In fact, the entire description of the Betti cone $\mathbb B_{\mathbb Q}(R)$ is known if $\mu(\m)=1$ (see \cite[Proposition 1.1]{BBEG12}). Therefore, for the rest of the article, we assume that $\mu(\m)\geq 2$. Our first goal is to prove that $\sk$ has a pure resolution if and only if the resolution is linear. We prepare it with the following lemma.

\begin{lemma}\label{quadraticlemma}
Let $R$ be a standard graded $\sk$-algebra with $\mu(\m)\geq 2$. If $\sk$ has a pure resolution over $R$, then $R$ is a quadratic algebra.
\end{lemma}

\begin{proof}
Let $R=S/I$, where $S=\sk[X_1,\dots,X_n]$ and $I\subset {(\m_S)}^2$ is a homogeneous ideal in $S$. Set $x_i=X_i+I$ in $R$. More generally, we use capital letters for polynomials in $S$ and use small letters for corresponding image in $R$.

Let $\ \cdots\to R(-1)^n\overset{\phi_1}\lrar R\to \sk\to 0$ be a minimal free resolution of $\sk$ over $R$ and $\Omega_i$ denote its $i^{th}$ syzygy module. Note that $\Omega_1\simeq \m$. Therefore $\phi_1(e_i)=x_i$ for all $i$, where $\{e_1, \ldots, e_n\}$ is a basis of $R(-1)^n$.

Take  a minimal generator $F\in I\setminus  \m_S I$ of $I$, and write $F=\sum_i G_iX_i$, where $G_i\in\m_S$.
Then we have $\phi_1(\sum g_ie_i)=f=0$ in $R$, and hence $\sum g_ie_i\in\Omega_2$. We claim that $\sum g_ie_i\notin\m\Omega_2$. 

Suppose $\sum g_ie_i\in\m\Omega_2$, then there exist $h_1,\ldots, h_r\in \Omega_2$ and $b_1,\dots,b_r\in \m$ such that $\sum g_ie_i=\sum b_jh_j$. Write $h_j=\sum_ih_{j_i}e_i$ for all $j$. Hence, $\sum_ih_{j_i}x_i=\phi_1(\sum_ih_{j_i}e_i)=0$. Thus, $\sum_iH_{j_i}X_i\in I$ for all $j$. 

Now, $\sum g_ie_i=\sum b_jh_j$ and the fact that $\{e_1, \ldots, e_n\}$ is a basis implies that $g_i=\sum_j b_jh_{j_i}$ for all $i$.
Thus, for all $i$, we have $G_i=\sum_j B_jH_{j_i} +G_i'$ for some $G_i'\in I$. This implies that $\sum_iG_iX_i=\sum_{i,j} B_jH_{j_i}X_i +\sum_iG_i'X_i$. Note that $\sum_{i,j} B_jH_{j_i}X_i +\sum_iG_i'X_i\in \m_SI$ follows from the fact that for all $j$, we have $ \sum_iH_{j_i}X_i\in I$, $B_j\in\m_S$ and $\sum_iG_i'X_i\in \m_SI$. Thus, $F\in \m_S I$ which is a contradiction.

Note that $x_je_i-x_ie_j\in \Omega_2$ and $\deg(x_je_i-x_ie_j)=2$. Since $\sk$ has a pure resolution and $\sum g_ie_i\in\Omega_2\setminus \m\Omega_2 $, we have $\deg(\sum g_ie_i)=2$. Therefore $\deg(F)=2$, and hence $I$ is minimally generated in degree $2$.
\end{proof}

In \cite[Remark 6(1)]{CDR13}, it is remarked that if $\beta_{2,j}^R(\sk)=0$ for every $j\neq2$, then $R$ is a quadratic algebra. We prove this in the previous lemma for the sake of completeness. From \cite[Remark 11]{CDR13}, it follows that quadratic complete intersections are Koszul algebras, which we use in the following theorem.

\begin{thm}\label{Koszul}
Let $R$ be a standard graded $\sk$-algebra with $\mu(\m)\geq 2$. Then  $\sk$ has a pure resolution if and only if $\sk$ has a linear resolution, i.e., $R$ is Koszul.
\end{thm}

\begin{proof}  If $\sk$ has a linear resolution, then the resolution is pure. For the converse, assume that $\sk$ has a pure resolution over $R$. Let $R=S/I$, where $S=\sk[X_1,\dots,X_n]$ and $I\subset \m_S^2$ is homogeneous ideal in $S$.  Then, by Lemma \ref{quadraticlemma}, $I$ is generated in degree $2$.  Let $f\in I\setminus \m_S I$ and $R'=S/\langle f\rangle $. Since $R'$ is a quadratic hypersurface, $R'$ is Koszul, i.e., $\sk$ has a linear resolution over $R'$. Let $\F: \cdots \overset{\phi_2}\lrar F_1\overset{\phi_1}\lrar F_0\lrar \sk\rar0$ be a minimal free resolution of $\sk$ over $R'$.
	
{\it Claim}:  
$\G :\cdots \overset{\psi_2}\lrar G_1\overset{\psi_1}\lrar G_0\lrar \sk\rar0$ is a minimal free resolution of $\sk$ over $R$, where $$\bar{\phi_i}=\phi_i\otimes_{R'} R \text {\ \  and \ \ }
\psi_i=\left[\begin{array}{cc}
\bar{\phi_i} & *\\0&*
\end{array}\right].$$
	
To prove the claim, we use induction on $i$. Note that $\phi_1$, $\psi_1$ are images of the matrix $\left[\begin{array}{cccc}
X_1&X_2&\cdots&X_n
\end{array}\right]$ in the respective rings, hence the claim is true for $i=1$. Now, assume that claim is true for $j\leq i$, i.e., assume that for all $j \leq i$ we have
$$\psi_j=\left[\begin{array}{cc}
\bar{\phi_j} & *\\0&*
\end{array}\right].$$

Let $\alpha = \left[\begin{array}{ccc}
\alpha_1 & \cdots & \alpha_{\beta_i}
\end{array}\right]^t$
 be a column of $\phi_i$, i.e., $\alpha\in \Omega_{i+1}^{R'}(\sk)\setminus \m_{R'} \Omega_{i+1}^{R'}(\sk)$. Note that since $\sk$ has a linear resolution over $R'$, we have $\alpha_i\in\m_{R'}\setminus\m_{R'}^2$. Furthermore, since $I$ is generated in degree $2$, $\deg(\alpha_i)=\deg(\bar{\alpha_i})=1$, where \, $\bar{}$ \,  denotes images in $R$.

 Now, $\phi_i\alpha=0$ implies $\left[\begin{array}{cc}
\bar{\phi}_i & *\\0&*
\end{array}\right]\left[\begin{array}{c}
\bar{\alpha}\\0
\end{array}\right]=0$, and hence $ \left[\begin{array}{c}
\bar{\alpha}\\0
\end{array}\right]\in \Omega_{i+1}^R(\sk)$. Moreover $\bar{\alpha}\notin\m_R\Omega_{i+1}^R(\sk)$,  since $\bar{\alpha_j}\in \m_R\setminus\m_R^2$ for all $j$ and $\Omega_{i+1}^R(\sk)\in\m_RG_i$. Therefore claim holds for $i+1$.

Since $\sk$ has a pure resolution, each nonzero entry of $\psi_i$ has the same degree. By the claim, these nonzero entries must be linear, and hence $\sk$ has a linear resolution.
\end{proof}

The above lemma gives us the following necessary condition for the equality $\mathbb B_{\mathbb Q}(R)= \mathbb B_{\mathbb Q}^{\rm{pure}}(R)$.  
\begin{thm}\label{thm:purityImpliesKoszul}
    Let $R$ be a standard graded $\mathsf k$-algebra such that the extremal rays of $\mathbb B_{\mathbb Q}(R)$ are spanned by the Betti tables of pure $R$-modules.  Then $R$ is a Koszul algebra.
\end{thm}
\begin{proof}
    By Theorem~\ref{SomeExtremalRays}, we know that $\beta^R(\sk)$ spans an extremal ray of $\mathbb B_{\BQ}(R)$. Thus, by hypothesis, $\mathsf k$ has a pure resolution. Now the  result follows from Theorem~\ref{Koszul}.
\end{proof}

Next, we provide yet another necessary condition in for the equality $\mathbb B_{\mathbb Q}(R)= \mathbb B_{\mathbb Q}^{\rm{pure}}(R)$. We prove that when the equality holds,  all finitely generated modules over $R$ have rational Poincar\'e series sharing a common denominator. 

\begin{defn}[\cite{Ro05}, Definition 2.1]
    A standard graded $\sk$-algebra $R$ is said to be \emph{good} (in the sense of Roos) if there exists a polynomial $g(z)\in \mathbb Z[z]$ such that for every finitely generated graded $R$-module $M$ we have $\mathcal P_M^R(z) g(z) \in \mathbb Z[z]$.
\end{defn}
Note that in particular, over good algebras, every finitely generated graded module has rational Poincar\'e series. 

\begin{prop}\label{prop:linearSyzygy}
Let $R$ be a Koszul $\sk$-algebra with $H_R(z)= f(z)/(1-z)^d$, and $M$ be a pure $R$-module. Then 
\begin{enumerate}[{\rm (a)}]
    \item there exists $i\geq 0$ such that $\Omega_i^R(M)$ has linear resolution. 
    \item $\mathcal P_M(z)$ is rational with denominator $f(-z)$.
\end{enumerate}
\end{prop}
\begin{proof}
(a) Since $R$ is Koszul, by Theorem \ref{KoszulAEP}, we have $\reg_R(M)=s<\infty$. Then there exists $i\geq 0$ such that $\beta^R_{i,i+s}(M) \neq 0$. Note that $\Omega_i^R(M)$ has pure resolution. Since $\reg_R(M)=s$, we see that $\beta^R_{i+j,i+s+j+k}(M)=0$ for all $j \geq 0, k\geq 1$. Therefore, the resolution of $\Omega_i^R(M)$ is linear.

(b) Since $\Omega_i^R(M)$ has a linear resolution, we get $\mathcal P_{\Omega_i^R(M)}(z)= H_{\Omega_i^R(M)}(-z)/H_R(-z)$. Suppose that $H_{\Omega_i^R(M)}(z)=g(z)/(1-z)^m$. Then we get $\mathcal P_{\Omega_i^R(M)}(z)=(1+z)^{d-m}g(-z)/f(-z)$. Since $\mathcal P_M(z) - \mathcal P_{\Omega_i^R(M)}(z)$ is a polynomial, we get that $\mathcal P_M(z)$ is rational with denominator $f(-z)$. 
\end{proof}

\begin{cor}\label{cor:commonDenominator}
Let $R$ be a Koszul $\sk$-algebra with $H_R(z) =f(z)/(1-z)^d$. Then for every $R$-module $M$ with $\beta^R(M)\in \mathbb B ^{{\rm{pure}}}_{\mathbb Q}(R)$, we have that $\mathcal P_M(z)$ is rational with denominator $f(-z)$.
\end{cor}
\begin{proof}
We may assume that $M\neq 0$. Since $\beta^R(M) \in  \mathbb B ^{\rm{pure}}_{\mathbb Q}(R)$, there exist pure $R$-modules $M_1,\ldots, M_r$ and $c_1,\ldots, c_r \in \mathbb Q_{>0}$ such that $\beta^R(M)= \sum\limits_{i=1}^r c_i\beta^R(M_i)$. Hence, the result follows, since by Proposition \ref{prop:linearSyzygy}, each $\beta^R(M_i)$ has rational Poincar\'e series with denominator $f(-z)$.
\end{proof}

As an immediate consequence of Corollary \ref{cor:commonDenominator} and Theorem \ref{thm:purityImpliesKoszul}, we obtain the following: 
\begin{thm}\label{thm:good} 
Let $R$ be a standard graded $\sk$-algebra with $H_R(z)=f(z)/(1-z)^d$ such that $\mathbb B_{\mathbb{Q}}(R)=\mathbb B_{\mathbb{Q}}^{\rm{pure}}(R)$. Then for every $R$-module $M$, $\mathcal P_M(z)$ is rational with denominator $f(-z)$.
In particular, if $\mathbb B_{\mathbb{Q}}(R)=\mathbb B_{\mathbb{Q}}^{\rm{pure}}(R)$, then $R$ is good.
\end{thm}

Two more classes of standard graded algebras are closely related to good algebras: (a) Absolutely Koszul algebras and (b) Algebras with Backelin--Roos property. An \emph{absolutely Koszul algebra} $R$ is an algebra over which every finitely generated module has finite \emph{linearity defect}. An algebra $R$ is said to satisfy \emph{Backelin--Roos property} if there exists a \emph{Golod} map $\varphi:Q \to R$ for some comlpete intersection ring $Q$. More details on the same can be found in \cite{CINR15}. It is known that if $R$ is Koszul, then we have the following chain of implications (see \cite{HI07}): 
\begin{center}
    $R$ satisfies Backelin--Roos property $\Longrightarrow R$ is Absolutely Koszul $\Longrightarrow R$ is good.
\end{center}
However, it is not known whether the above implications are strict or not. Thus, in view of Theorem \ref{thm:good}, the following natural questions arise:
\begin{question}
    Let $R$ be a standard graded $\sk$-algebra such that $\mathbb B_{\mathbb{Q}}(R)=\mathbb B_{\mathbb{Q}}^{\rm{pure}}(R)$.
\begin{enumerate}[{\rm (a)}]
        \item Is $R$ absolutely Koszul?
        \item Does $R$ satisfy Backelin--Roos property?
    \end{enumerate}
\end{question}

\section{Purity of Extremal Rays}\label{sec:purity-special}

In this section, we focus on some specific classes of standard graded algebras and study the equality of the pure and Betti cones for them. 

\subsection{Purity of Extremal Rays and the Artinian Property}

\begin{lemma}\label{syzygyLemma} Let $R$ be a standard graded $\sk$-algebra with $\mu(\m)=n$, and
$M$ an $R$-module. Fix $i\in\BZ_{\geq0}$, let $\Omega=\Omega_i^R(M)$ and $j$ be the smallest integer such that $\beta_{i,j}^R(M)\neq0$. Then $\beta_{i+1,j+1}^R(M)\leq n\beta_{i,j}^R(M)$ and equality holds if and only if\, $ \Omega^{(j)}\simeq \sk(-j)^{\beta_{i,j}^R(M)}$.
\end{lemma}
\begin{proof} Note that $\beta_{i+1,j+1}^R(M)=\beta_{1,j+1}^R(\Omega)=\beta_{1,j+1}^R(\Omega^{(j)})$, where the last equality follows from Remark~\ref{rem:generators-up-to-degree-j}. Since $\beta_{i,k}^R(M)=0$ for all $k<j$, we see that $\Omega^{(j)}$ is generated in degree $j$, and hence the shifted module $\Omega^{(j)}(j)$ is generated in degree $0$. Therefore we get $\beta_{1,1}^R(\Omega^{(j)}(j))\leq n\beta_{0,0}^R(\Omega^{(j)}(j))$ and equality holds if and only if $\Omega^{(j)}(j)\simeq \sk^{\oplus\beta_{1,1}^R(\Omega^{(j)}(j))}$, by Remark~\ref{rem:directSumofR/m^j}. Hence the proof is complete, since $\beta_{i,k}^R(\Omega^{(j)}(j))=\beta_{i,k+j}^R(\Omega^{(j)})$ for all $i,k\in \BZ$, and $\beta_{i+1,j+1}^R(M)=\beta_{1,j+1}^R(\Omega^{(j)})$. 
\end{proof}

\begin{lemma}\label{mainLemma}
Let $R$ be a standard graded $\sk$-algebra with $\mu(\m)=n$ and $M$ an $R$-module generated in degree $0$. Suppose $\beta_{1,j}^R(M)=0$ for all $j<s$ and $\beta_{2,s+1}^R(M)=n\beta_{1,s}^R(M)\neq 0$. Let ${\bf d}=(0,s,s+1,d_3,d_4,\ldots)$ be a degree sequence and $M'$ a pure $R$-module of type ${\bf d}$. If $\beta^R(M)-c\beta^R(M')\in\mathbb{B}_{\BQ}(R)$ for some $c\in \BQ_{>0}$, then $\Omega^R_1(M')\simeq \sk(-s)^{\oplus\beta^R_{1,s}(M')}$.
\end{lemma}

\begin{proof}
Since $\beta^R_{1,j}(M')=0$ for $j<s$, by Lemma \ref{syzygyLemma}, we know that $\beta^R_{2,s+1}(M')\leq n\beta^R_{1,s}(M')$. Therefore, we have $$\beta^R_{2,s+1}(M)-c\beta^R_{2,s+1}(M')=n\beta^R_{1,s}(M)-c\beta^R_{2,s+1}(M')\geq n\left(\beta^R_{1,s}(M)-c\beta^R_{1,s}(M')\right).$$
On the other hand, since $\beta^R(M)-c\beta^R(M')\in\mathbb{B}_{\BQ}(R)$ and $\beta^R_{1,j}(M)-c\beta^R_{1,j}(M')=0$ for all $j<s$, by Remark \ref{bettiTableremark} and Lemma \ref{syzygyLemma}, we get
$\beta^R_{2,s+1}(M)-c\beta^R_{2,s+1}(M')\leq n\left(\beta^R_{1,s}(M)-c\beta^R_{1,s}(M')\right).$  Thus, we see that  $\beta^R_{2,s+1}(M)-c\beta^R_{2,s+1}(M')= n\left(\beta^R_{1,s}(M)-c\beta^R_{1,s}(M')\right)$, and hence we get  
$\beta^R_{2,s+1}(M')=n\beta^R_{1,s}(M')$. Thus, by Remark~\ref{rem:directSumofR/m^j}, we get $\Omega_1^R(M')\simeq \sk(-s)^{\oplus\beta^R_{1,s}(M')}$.
\end{proof}

Recall that an Artinian standard graded $\sk$-algebra $R$ is said to be \emph{level} if its socle $\langle 0\rangle \colon_R \m$ is concentrated in a single degree. A Cohen--Macaulay $\sk$-algebra $R$ of dimension $d$ is said to be \emph{level} if some (and hence all) its Artinian reductions $R/\langle x_1, \ldots, x_d\rangle$ by linear forms $x_1, \ldots, x_d$ is an Artinian level algebra. 

\begin{thm}\label{thm:depthZeroImpliesArtinianLevel}
Let $R$ be a standard graded $\sk$-algebra such that $\depth(R)=0$. If $\mathbb B_{\mathbb Q}(R)=\mathbb B _{\mathbb Q}^{\rm pure}(R)$, then $R$ is Artinian and level.
\end{thm}

\begin{proof} 
Let $0\neq x\in\soc(R)$, with $\deg(x)=s$. We first claim that $\m^{s+1}=0$. Suppose this is not true. Let $M=R/(\langle x\rangle+\m^{s+1})$. Since $\Omega_1^R(M)\simeq \sk_j(-s)\oplus \m^{s+1}$, we get the following Betti table
$$\beta^R(M)=\begin{tabular}{|c |ccccc| }
\hline
\backslashbox{$j$}{$i$}& 0 & 1 & 2 &3&$\cdots$\\
\hline
0 & 1 & - & -&-&$\cdots$\\
$\vdots$ &$\vdots$ &$\vdots$ &$\vdots$ &$\vdots$&$\vdots$ \\
s-1 & - & 1 & $\beta_{1,1}^R(\sk)$ & $\beta_{2,2}^R(\sk)$&$\cdots$\\
s & - & $\mu(\m^{s+1})$ & $*$ & $*$&$\cdots$\\
\vdots&-&-&$\vdots$&\vdots & $\ddots$\\
\hline
\end{tabular}\ .$$ 
	
We now show that $\beta^R(M)$ cannot be written as a positive $\BQ$-linear combination of Betti tables of pure $R$-modules. Suppose there exists $R$-modules $M_i$ and $c_i\in\BQ_{>0}$ such that $\beta^R(M)=\sum_ic_i\beta^R(M_i)$. Therefore, we get $H_M(z)=\sum_ic_iH_{M_i}(z)$. Since $\deg(H_M(z))=s$, and $c_i>0$, we see that $H_{M_i}(z)$ is a polynomial with $H_{M_i}(z)\leq s$ for all $i$.
	
Observe that $\beta_{2,s+1}^R(M)= \beta^R_{1,1}(\sk)\neq 0$ and $\beta_{3,s+2}^R(M)= \beta^R_{2,2}(\sk)\neq 0$. Hence there exists some $i$
such that $M_i$ has a pure resolution of type $(0,s,s+1,s+2,d_4,d_5,\dots)$. By Lemma \ref{mainLemma}, we get $\Omega_1^R(M_i)\simeq \sk(-s)^{\oplus\beta_{1,s}^R(M_i)}$. 
This implies that $H_{\Omega_1^R(M_i)}(z)=\beta_{1,s}^R(M_i)z^s$.	By the additivity of the Hilbert series, we get	
$$H_{M_i}(z)-\beta^R_{0,0}(M_i)H_{R}(z)-\beta^R_{1,s}(M_i)z^s=0.$$  Since $\deg(H_{M_i}(z))\leq s$, $\deg(H_R(z))>s$, by comparing the coefficient of $z^{s+1}$, we get $\beta^R_{0,0}(M_i)=0$ which is a contradiction to $M_i\neq 0$. This contradiction shows that $\m^{s+1}=0$, and hence $R$ is Artinian. Moreover, we have $0\neq \m^s\subset \soc(R)$. On the other hand, if $0\neq y\in \soc(R)$ with $\deg(y)=t<s$,  then the same argument shows that $\m^{t+1}=0$. Thus, $t\geq s$ proving the $R$ is level. 
\end{proof}

The following example shows that a Koszulness is not sufficient for the equality $\mathbb B_{\mathbb Q}(R)=\mathbb B _{\mathbb Q}^{\rm pure}(R)$ to hold. 
\begin{example}
Let $R=\sk[X,Y,Z]/\langle X^2,Y^2,Z^2,XZ,YZ\rangle$. Then $R$ is Koszul since it is a quadratic monomial algebra (see \cite[Remark 6(2)]{CDR13}). Since $\ov{Z}$ and $\ov{XY}$ are nonzero elements in socle of $R$. By the theorem above, we see that $\mathbb B_{\mathbb Q}(R)\neq \mathbb B _{\mathbb Q}^{\rm pure}(R)$.
\end{example}

There are no known examples of non-Cohen--Macaulay algebras $R$ satisfying $\BB_{\mathbb{Q}}(R)=\BB^{\rm{pure}}_{\mathbb{Q}}(R)$. Theorem~\ref{thm:depthZeroImpliesArtinianLevel} motivates us to propose the following.

\begin{conjecture}
    Let $R$ be a standard graded $\sk$-algebra such that $\BB_{\mathbb{Q}}(R)=\BB^{\rm{pure}}_{\mathbb{Q}}(R)$. Then $R$ is a Cohen--Macaulay level algebra.
\end{conjecture}

\subsection{Rings with Pairs of Exact Zerodivisors}
\label{HVector}
\hfill{}\\
\vspace{-10pt}

\begin{defn}{\rm
Let $R$ be a standard graded $\sk$-algebra and $a,b\in R_1$. We say that $\{a,b\}$ is a \emph{pair of linear exact zerodivisors} if $\langle 0 \rangle :_R \langle a \rangle =\langle b\rangle$ and $\langle 0 \rangle :_R \langle b \rangle =\langle a\rangle$.\index{pair of linear exact zerodivisors}
}\end{defn}  

\begin{rmk}\label{linearBetti}{\rm Let $R$ be a standard graded $\sk$-algebra with $H_R(z)=\dfrac{f(z)}{(1-z)^d}$. 
\begin{enumerate}[{\rm (a)}]
    \item If $R$ has a pair $\{a,b\}$ of linear exact zerodivisors, then $\beta^R_{i,i}(R/\langle a\rangle)=1$ for all $i$, and hence $\mP_{(R/\langle a\rangle)}(z)=1/(1-z)$.  Observe that, by Remark~\ref{rmk:Hilbert-series-etc}(c), we get that $(1+z)$ divides $f(z)$.
    \item  If $R$ is a generic Gorenstein Artin standard graded $\mathsf k$-algebra of socle degree $3$, then $R$ has a pair of linear exact zerodivisors (see \cite[Remark 4.3]{HS11}).
\end{enumerate}
}\end{rmk}
\begin{thm}\label{notSufficient} Let $R$ be a standard graded $\sk$-algebra with $H_R(z)={f(z)}/{(1-z)^d}$, where $\dim(R)=d$. Suppose $\mathbb B_{\mathbb Q}(R)= \mathbb B_{\mathbb Q}^{\rm pure}(R)$. If $R$ has a pair of linear exact zerodivisors, then $\deg(f(z))\leq2$.
\end{thm}

We prepare the proof of Theorem \ref{notSufficient} with a few lemmas as follows.

\begin{lemma}\label{exactZD}
Let $R$ be a standard graded $\sk$-algebra, $\{a,b\}$ a pair of linear exact zerodivisors and $M=\dfrac{R}{\langle a\rangle +\m^2}$.  Then $\beta^R_{i,i}(M)=1$ for all $i\geq 0$.
\end{lemma}

\begin{proof}
Since $M$ is a cyclic module, generated in degree $0$ and $\ann_R(M)=\langle a\rangle +\m^2$, we have $\beta^R_{0,0}(M)=1$ and $\Omega^R_1(M)=\langle a\rangle +\m^2$.  We know that $\beta^R_{i,j}(\Omega_1^R(M))=\beta^R_{i+1,j}(M)$. Since $\deg(a)=1$ and $\deg(x)\geq 2$ for all nonzero $x\in \m^2$, by Remark \ref{rem:generators-up-to-degree-j}, it is enough to prove that $\beta^R_{i,i+1}(\langle a\rangle )=1$ for all $i\geq0$. But this follows from the fact that the minimal free resolution of $\langle a\rangle $ is of the form
$\cdots\overset{a}\longrightarrow R(-1)\overset{b}\longrightarrow R\overset{a}\longrightarrow \langle a \rangle \longrightarrow0.$
\end{proof}

\begin{lemma}\label{equalBetti}
Let $R$ be a standard graded $\sk$-algebra such that $\mathbb B_{\mathbb Q}(R)=\mathbb B _{\mathbb Q}^{\rm pure}(R)$. Suppose $M$ is an $R$-module generated in degree $0$ with $\beta^R_{i,i}(M)\neq0$ for all $i$. Then there exist $c\in\BQ_{>0}$, and a pure $R$-module $M'$ of type $(0,1,2,\dots)$ such that $\beta^R_{i,i}(M)=c\beta^R_{i,i}(M')$ for $i\gg0$.
\end{lemma}

\begin{proof}
 Remark~\ref{bettiTableremark}(ii) implies that $\beta^R(M)=\sum_{i=1}^{n} c_i\beta^R(M_i)$, where $c_i\in \BQ_{>0}$ and $M_i$'s are pure $R$-modules of type ${\bf d}_i$ with ${\bf d}_i\neq {\bf d}_j$ for $i\neq j$. Therefore, there is a unique $j$ for which $M_j$ has a pure resolution of type $(0,1,2,\ldots)$, and hence $\beta^R_{i,i}(M_k)=0$ for all $k\neq j$ and $i\gg0$. Thus, $M'=M_j$ is the desired module.
\end{proof}

\begin{proof}[Proof of Theorem \ref{notSufficient}]
Let $\{a,b\}$ be a pair of linear exact zerodivisors. Set $M=\frac{R}{\langle a\rangle +\m^2}$. By Lemma \ref{exactZD}, we get $\beta^R_{i,i}(M)=1$ for all $i\geq 0$. By hypothesis, there exist  pure $R$-modules $M_1, \ldots, M_n$ and $c_1, \ldots, c_n\in \BQ_{>0}$ such that $\beta^R(M)=\sum_{i=1}^{n} c_i\beta^R(M_i)$. 
By  Remark~\ref{bettiTableremark}(ii), we may assume that if $M_i$ is a pure $R$-module of type ${\bf d}_i$, then ${\bf d}_i\neq {\bf d}_j$ for all $i\neq j$.
By Lemma~\ref{equalBetti}, we see that for some $j$, $M_j$ has a linear resolution with $c_j\beta^R_{i,i}(M_j)=\beta^R_{i,i}(M)=1$ for $i\gg0$. 
	
Note that $H_M(z)=1+mz$ for some $m\in\BN$. Thus, from the equality $H_M(z)=\sum_{i=1}^{n} c_i H_{M_i}(z)$, we see that $H_{M_j}(z)=r+sz$ for some $r,s\in\BZ_{\geq0}$. Note that since $M_j \neq0$, we have $r+sz \neq 0$. Now, by  Remark~\ref{rmk:Hilbert-series-etc}(c), we have $$\mathcal P_{M_j}(-z)=\dfrac{H_{M_j}(z)}{H_R(z)}=\frac{(r+sz)(1-z)^d}{f(z)}.$$

Since $\beta_{i,i}^R(M_j)=1/c_j$ for $i\gg0$, we see that $P_{M_j}(-z)=\dfrac{h(z)}{1+z}$ with $h(-1)\neq 0$. Hence $\dfrac{h(z)}{1+z}= \dfrac{(r+sz)(1-z)^d}{f(z)}$. Thus, we get $f(z)h(z)=(r+sz)(1-z)^d(1+z)$. Since $(1-z)$ does not divide $f(z)$, we see that $(1-z)^d$ divides $h(z)$. Therefore,  $f(z)$ divides $(r+sz)(1+z)$, and hence $\deg(f(z))\leq 2$.
\end{proof}

\subsection{Generic Gorenstein Algebras}
\hfill{}\\
\vspace{-10pt}

Our goal in this subsection is to characterize generic Gorenstein Artin standard graded $\sk$-algebras $R$ satisfying $\mathbb B_{\mathbb Q}(R)=\mathbb B_{\mathbb Q}^{{\rm{pure}}}(R)$. Let $R$ is Gorenstein Artin standard graded $\sk$-algebra with $\edim(R)=e\geq2$, and socle degree $s\geq 2$.

\begin{rmk}\label{rmk:Gorenstein-and-Koszul} 
We record some facts about the Koszul property of Gorenstein Artin algebras.
\begin{enumerate}[{\rm (a)}]
    \item If $R$ is a Koszul Gorenstein Artin $\sk$-algebra with $e=2$, then we have $s=2$. Indeed, if such an algebra $R=\sk[x,y]/I$ exists, then the fact that $R$ is Gorenstein forces its Hilbert series $H_R(z)=\sum_{i=0}^s h_iz^i$ to satisfy $h_i=h_{s-i}$ for all $i$. In particular, we have $s \geq 2$. Moreover, the Koszulness of $R$ implies that $I$ is a quadratic ideal. Since $R$ is Artinian, we have $\mu(I)\geq 2$, and hence $h_2\leq 1$. This forces $h_i\leq 1$ for all $i \geq 2$. Therefore, the equality $h_1=h_{s-1}=2$ forces $s=2$.  
    \item If $R$ is generic with $e \geq 3$ and $s=3$, then $R$ is Koszul (see \cite[Theorem 6.3]{CRV01}).
    \item If $R$ is generic with $e \geq 4$ and $s\geq 3$, then $R$ is not Koszul (see \cite[Corollary 3.6]{BNSR25}).
\end{enumerate}
\end{rmk}

In \cite{Gi13}, Gibbons proved that $\mathbb B_{\mathbb Q}(R)=\mathbb B_{\mathbb Q}^{{\rm{pure}}}(R)$ holds for the case $e\geq 2$ and $s=2$. 
Now, suppose that $R$ is generic with $e \geq 2, s\geq 3$. We claim that $\mathbb B_{\mathbb Q}(R)\neq \mathbb B_{\mathbb Q}^{{\rm{pure}}}(R)$. To see this, note that by Theorem \ref{thm:purityImpliesKoszul} and Remark \ref{rmk:Gorenstein-and-Koszul}, we must have $e \geq  3$ and $s=3$. Thus, by Remark \ref{linearBetti} and Theorem \ref{notSufficient}, it follows that $\mathbb B_{\mathbb Q}(R)\neq \mathbb B_{\mathbb Q}^{{\rm{pure}}}(R)$. We therefore have the following result.
\begin{thm}\label{thm:generic-Gorenstein}
    Let $R$ be a Gorenstein Artin standard graded $\mathsf k$-algebra with $\edim(R)\geq 2$ and socle degree $s$. Then  
    \begin{enumerate}[{\rm (a)}]
        \item If $s=2$, then $\mathbb B_{\mathbb Q}(R)= \mathbb B_{\mathbb Q}^{{\rm{pure}}}(R)$.
        \item If $s \geq 3$ and $R$ is generic, then $\mathbb B_{\mathbb Q}(R)\neq  \mathbb B_{\mathbb Q}^{{\rm{pure}}}(R)$.
    \end{enumerate}
\end{thm}

\section{Purity of Extremal Rays and Minimal Multiplicity}\label{sec:purity-and-minimal-multiplicity}

As noted in the introduction section, given any standard graded algebra $R$, the Betti tables of shifts of the modules of the form (a) $R/\m^j$, where $j \geq 0$, and (b) the modules $M$ with pure resolution satisfying $\codim(M)=\pdim_R(M)$ always span an extremal ray in $\mathbb B_{\mathbb Q}(R)$. If $R$ is a polynomial ring, then due to the work of Eisenbud--Schreyer \cite{ES09}, we know that (a) and (b) gives a complete set of extremal rays of $\mathbb B_{\mathbb Q}(R)$. In this section, we characterize all standard graded algebras $R$ having (a) and (b) to be the only extremal rays of $\mathbb B_{\mathbb Q}(R)$. 

\begin{lemma}\label{lem:divisibility}
    Let $R$ be a standard graded $\sk$-algebra with $\edim(R)\geq 2$, Hilbert series $H_R(z)= \sum_{i\in \mathbb Z} h_i z^i$, and $I$ be a nonzero ideal of $R$ generated in degree $i$. Suppose that
 \begin{enumerate}[{\rm (a)}]
            \item $\mathbb B_{\mathbb Q}(R)=\mathbb B_{\mathbb Q}^{{\rm{pure}}}(R)$.
            \item The extremal rays of $\mathbb B_{\mathbb{Q}}(R)$ are spanned by Betti tables of shifts of $R$ and $R/\mathfrak m^i, i \geq 1$. 
            \end{enumerate}
         Then $R$ is Artinian, $I$ has a linear resolution, and $\beta^R(I)= \dfrac{\mu(I)}{\mu(\m^i)}\beta^R(\m^i)$.\\
         In particular, $\beta_{1, i+1}(I)= {\mu(I)} \left( h_1 - \dfrac{h_{i+1}}{h_i}\right)$.  
\end{lemma}
\begin{proof}
    If $R$ contains a nonzerodivisor $x$, then by Remark \ref{rem:R-mod-regular-seq-are-extremal}, we see that $\beta^R(R/\langle x \rangle)$ spans an extremal ray of $\mathbb B_{\mathbb Q}(R)$, which is a contradiction to (b). This shows that $\depth(R)=0$. By Theorem \ref{thm:depthZeroImpliesArtinianLevel}, we get that $R$ is Artinian. 

    Observe that $\beta_{1,j}(R/\mathfrak m^i)\neq 0$ implies $j=i$. By condition (b), there exist $c_1\in \mathbb Q_{\geq 0}, c_2\in \mathbb Q_{>0}$ such that $\beta^R(R/I)=c_1 \beta^R(R)+c_2 \beta^R(R/\mathfrak m^i)$. Thus, $\beta^R(I)=c_2\beta^R(R/ \m^i)$. This forces $c_2=\mu(I)/\mu(\m^i)=\mu(I)/h_i$. Note that by Theorem \ref{thm:purityImpliesKoszul}, we have that $R$ is Koszul.  Hence $\m^i$ has a linear resolution, which forces $I$ to have a linear resolution. 

    Since $\m^i(i)$ is generated in degree zero and has a linear resolution, we have $H_{\mathfrak m^i(i)}(z)=H_R(z)\mathcal{P}_{\mathfrak m^i(i)}(-z)$. Therefore, 
    \[\mathcal{P}_{\mathfrak m^i(i)}(z)=\dfrac{h_i-h_{i+1}z+\cdots}{1-h_1z+h_2z^2-h_3z^3+\cdots}=h_i+(h_1h_i-h_{i+1})z+\cdots,\]
which shows $\beta_{1,i+1}(\m^i) = h_1h_i-h_{i+1}$. This proves the equality $\beta_{1,i+1}(I)=\mu(I)\left(h_1 - \dfrac{h_{i+1}}{h_i}\right)$.
\end{proof}

\begin{lemma}\label{lem:mingenswithNonzeroproduct}
    Let $R$ be a standard graded $\mathsf k$-algebra with $\edim(R)=n\geq 2$ such that $\m^2 \neq 0$. Then there exists a minimal generating set $\{x_1,\ldots,x_n\}$ of $\m$ such that $x_ix_j\neq 0$ for some $i \neq j$.
\end{lemma}
\begin{proof}
Let $S=\{x_1,\ldots,x_n\}$ be any minimal generating set of $\m$. If $S$ does not satisfy the desired property, then we have $x_i x_j = 0$ for all $i \neq j$. Since $\m^2 \neq 0$, we get $x_i^2 \neq 0$ for some $i$. Then for $j\neq i$, the generating set $\left(S \setminus \{x_j\}\right) \cup \{x_i+x_j\}$ has the desired property.
\end{proof}

\begin{thm}\label{thm:converseIntheArtiniancase}
    Let $R$ be a standard graded $\sk$-algebra with $\edim(R)\geq 2$ such that 
 \begin{enumerate}[{\rm (a)}]
            \item $\mathbb B_{\mathbb Q}(R)=\mathbb B_{\mathbb Q}^{{\rm{pure}}}(R)$.
            \item The extremal rays of $\mathbb B_{\mathbb{Q}}(R)$ are spanned by Betti tables of shifts of $R$ and $R/\mathfrak m^i, i \geq 1$. 
            \end{enumerate}
     Then $H_R(z)=1+nz$ for some $n \in \mathbb Z$, and hence $R$ is level.
\end{thm}
\begin{proof}
   By Lemma \ref{lem:divisibility}, we get that $R$ is Artinian. Suppose that $\m^2\neq 0$. Then by Lemma \ref{lem:mingenswithNonzeroproduct}, there exist homogeneous minimal generators $x, y\in R_1$ of $\m$ such that $xy\neq 0$. Let $M= R/\langle x,y\rangle$.  Since $\Omega_1(M)= \langle x, y \rangle$ is generated in degree one, Lemma \ref{lem:divisibility} with $I=\langle x, y\rangle$ shows that $M$ has a linear resolution, and $\beta_2(M)= 2 \left( h_1 - \dfrac{h_2}{h_1}\right)$. 
   
   Note that for any homogeneous element $z \in R$, we have $\Omega_1(\langle z \rangle)= \ann(z)$. Hence, Lemma \ref{lem:divisibility} applied to $ \langle x\rangle$ and $\langle y\rangle$ shows that  $\ann(x)$ and $\ann(y)$ are generated in degree one with $\mu(\ann(x))=\mu(\ann(y))=h_1 - \dfrac{h_2}{h_1}=r$ (say). Let $\{f_1,\ldots,f_r\}$ and  $\{ g_1,\ldots,g_r\}$, where $f_i, g_j \in R_1$ for all $1\leq i, j\leq r$ be minimal homogeneous generating sets of $\ann(x)$ and $\ann(y)$, respectively.

    Consider a minimal free resolution $\cdots \to R e_1\oplus R e_2\xrightarrow[]{\phi_1} R \to M \to 0$ of $M$, where $\phi_1(e_1)=x$ and $\phi_1(e_2)=y$. Then observe that $\{f_ie_1, g_je_2 \mid 1 \leq i, j \leq r\}$ is a part of a minimal generating set of $\Omega_2(M)$, since each $f_ie_1, 1 \leq i \leq r$ and each $g_je_2, 1\leq j \leq r$ is a minimal generator of $\Omega_2(M)$.
    Now, note that $ye_1-xe_2 \in \Omega_2(M)\setminus  \langle\{f_ie_1, g_je_2 \mid 1 \leq i, j \leq r\}\rangle$, since $xy\neq 0$. Therefore, $\mu(\Omega_2(M))=\beta_2(M)> 2r= 2 \left(h_1 - \dfrac{h_2}{h_1}\right)= \beta_2(M)$, which is a contradiction. Thus, we must have $\m^2=0$. This completes the proof.
\end{proof}

\begin{defn}
    The \emph{fibre cone} of an ideal $I \subset R$ is defined as $ \mathcal F(I)= \bigoplus\limits_{i=0}^{\infty} \dfrac{I^i}{\m I^i}$. In particular, if $I=\m$, then $\mathcal{F}(I)=\mathcal F({\m})=G_{\mathfrak m}(R)$.
\end{defn}

\begin{defn}
    Let $I, J\subset R$ be ideals. Then $J$ is said to be a {reduction} of $I$ if $J \subset I$ and $I^{r+1}=JI^r$ for some $r \in \mathbb N$. A reduction that is minimal with respect to containment of ideals is called as a \emph{minimal reduction}. 
\end{defn}

\begin{rmk}\label{rmk:MinimalReductionandFibreCone}{\rm In this remark, we recall some known facts about minimal reductions. 
\begin{enumerate}[{\rm (a)}]
        \item Every ideal has a minimal reduction. (See \cite[Section 2, Theorem 1]{NR54})
        \item If $J$ is a minimal reduction of $I$, then for every ideal $K$ with $J\subset K \subset I$, we have $\mu(K)\geq \mu(J)$. (See \cite[Section 2, Lemma 3]{NR54})
        \item If $J$ is a minimal reduction of $I$, then $\mu(J)=\dim (\mathcal F(I))$. In particular, if $I$ is a minimal reduction of $\m$, then $\mu(I) = \dim(G_{\mathfrak m}(R))=\dim(R)$. (See \cite[Section 4, Theorem 1]{NR54})
    \end{enumerate}
}\end{rmk}

The following result is known, but we give a proof due to its immediate implications to our next theorem.
\begin{prop}\label{prop:HilbertSeries}
    Let $R$ be a standard graded $\sk$-algebra with $\dim(R)=d$ and $\edim(R)=n$. If $ x_1,\ldots, x_t$ is a linear $R$-regular sequence such that $\m^2=\langle \underline x \rangle \m$, then $t=d$, and $R$ is Cohen--Macaulay with $H_R(z)= \dfrac{1+(n-d)z}{(1-z)^d}$. 
\end{prop}
\begin{proof}
    Let $\underline x$ denote $x_1,\ldots, x_t$. Since $\m^2=\langle \underline x \rangle \m$, the ideal $\langle \underline x\rangle$ is a reduction of $\m$. By Remark \ref{rmk:MinimalReductionandFibreCone}, we have that $t=\mu(\langle \underline x\rangle)\geq \dim\left(G_{\mathfrak m}(R)\right) = \dim(R)=d$. Therefore, we must have $t=d$, which proves that $R$ is Cohen--Macaulay.
    
    Since $\underline x$ is a linear regular sequence, we have $H_{R/\langle \underline x \rangle}(z)=(1-z)^d H_R(z)$. Therefore, to complete the proof, it is enough to show that $H_{R/\langle \underline x \rangle}(z)=1 + (n-d)z$. Since $\m^2 =\langle \underline x \rangle \m$, we see that $\dim_{\mathsf k} \left(\left(R/\langle \underline x \rangle\right)_i\right)=0$ for all $i\geq 2$. Since $\underline x$ is a linear regular sequence, we get $\dim_{\mathsf k} \left(\left(R/\langle \underline x \rangle\right)_1\right)=\dim_{\mathsf k}(R_1)- d = n-d$. 
\end{proof}

\begin{thm}\label{thm:dim-at-most-1-for-smallest-set-of-extremals}
Let $R$ be a standard graded  $\sk$-algebra with $\dim(R)=d$ and $\edim(R)=n$ such that $\mathbb B_{\mathbb Q}(R)=\mathbb B _{\mathbb Q}^{\rm pure}(R)$. Suppose that $R$ is not a polynomial ring and that the extremal rays of $\mathbb B_{\mathbb Q}(R)$ are spanned by Betti tables of shifts of the following modules:
\begin{enumerate}[{\rm (a)}]
            \item $R/\m^i$, where $i\geq 0$
            \item Pure modules $M$ of finite projective dimension with $\codim(M)=\pdim(M)$.
\end{enumerate}
Then $R$ is a Cohen--Macaulay, level algebra of dimension at most one, with $H_R(z)= \dfrac{1+(n-d)z}{(1-z)^d}$.

Furthermore, when $\dim(R)=1$, the Betti tables of type (b) are Betti tables of shifts of $\beta^R(R)$ and $
\beta^R(R/\langle x \rangle)$, where $x$ is a linear nonzerodivisor in $R$.
\end{thm}
\begin{proof}
  Since $\mathbb B_{\mathbb Q}(R)=\mathbb B _{\mathbb Q}^{\rm pure}(R)$, by Theorem~\ref{thm:purityImpliesKoszul}, we see that $R$ is Koszul. 
     If $\depth(R)=0$, then by Theorem \ref{thm:depthZeroImpliesArtinianLevel}, we get that $R$ is Artinian and level. Therefore, by Theorem \ref{thm:converseIntheArtiniancase}, we get $H_R(z)=1+nz$, as required. 

    Now suppose that $\depth(R)=t\geq 1$, and consider a maximal linear $R$-regular sequence $x_1,\ldots,x_t$, denoted $\underline x$. We first claim that $\m^2= \langle \underline x \rangle \m$. If the claim is true, by Proposition \ref{prop:HilbertSeries} we would get that $R$ is Cohen--Macaulay, level algebra with the required Hilbert series.
    Suppose $y\in R_2$. We show that $y \in \langle \underline{x} \rangle$. 
    
    Set $M=R/\langle \underline x, y\rangle$. Then $\Omega_1(M)\simeq \langle \underline x, y\rangle$, and hence $\beta_{1,j}(M)=0$ for all $j \geq 3$. Thus, by our hypothesis about the extremal rays of $\mathbb B_{\mathbb Q}(R)$, we may write 
    $$\beta^R(M)= a \beta^R(R/\m) + b \beta^R(R/\m^2) + c \beta^R(N_0) + \sum_{j=1}^r d_j\beta^R(N_j),$$ where $a,b,c, d_j \in \mathbb Q_{\geq 0}$, $N_0$ is pure of type $(0,1,\ldots, t)$ with $\codim(N_0)=\pdim(N_0)$, and each $N_j$ is a pure module of finite projective dimension with $\codim(N_j)=\pdim(N_j)$. Furthermore, we may assume that $N_0, N_1, \ldots, N_r$ have different types.

    Since $\underline x$ is a linear regular sequence, the ideal $\langle \underline{x} \rangle$ has a linear resolution of type $(1,2,3,\ldots,t)$. 
    Hence, by  Remark \ref{rem:generators-up-to-degree-j}, we have $\beta_{i,i}(M)=\beta_{i-1,i-1}(\langle \underline x \rangle)=0$, for all $i>t$ and $\beta_{i,i}(M)={\binom{d}{i}
}$ for $0 \leq i \leq t$. This forces $a=0$.
    Since $\depth(R)=t$, given any $R$-module $L$, if $\pdim(L)<\infty$, then $\pdim(L)\leq t$. Therefore,  $\beta_{t,t}(N_j)\neq 0$ if and only if $j=0$. Since $\beta_{t,t}(M)=1\neq 0$, we see that $c \neq 0$. By Theorem \ref{thm:Herzog-Kuhl-in-general}, we have $\beta^R(N_{0})=d \beta^R(R/\langle \underline{x} \rangle)$ for some $d \in \mathbb Q_{>0}$. Since $\beta_{0,0}(R/\langle \underline{x}\rangle)=\beta_{t,t}(R/\langle \underline{x}\rangle)$, it follows that $\beta(M)=c\beta^R(N_0)=cd\,\beta(R/\langle \underline{x}\rangle)$. Now, $\beta_{1,2}(R/\langle \underline{x}\rangle)=0$ implies $\beta_{1,2}(M)=0$, and hence $\langle \underline x, y\rangle = \langle \underline x\rangle$, i.e., $y \in \langle \underline x \rangle$.

    We now show that $\dim(R)$ must be equal to one. We do this by contradiction. Accordingly, suppose that $d\geq 2$. Let $x$ be a linear nonzerodivisor on $R$, and $L=R/(\langle x \rangle +\m^2)$. Writing $R=S/I$, where $S=\sk[x_1,\ldots, x_n]$, we see that $\reg_S(L)\leq 1$. Thus, by Theorem~\ref{Thm-reg-over-poly-ring}, we get $\reg_R(L)\leq 1$. Since $d\geq 2$, $\m^2 \not\subset \langle x \rangle$, and hence $\beta_{1,2}^R(L)\neq 0$. This forces $\reg_R(L)=1$. Note that since $x$ is a nonzerodivisor, by Remark~\ref{rem:generators-up-to-degree-j}, we have that $\beta_{i,i}^R(L)=0$ for all $i \geq 2$. This shows that $\Omega_2^R(L)$ is generated in degree 3, and has a linear resolution. 
    
    Since $x$ is a linear nonzerodivisor, the module $\langle x \rangle + \m^2$ has one minimal generator in degree 1 
    and $\dim_{\mathsf k}(R_2)-\dim_{\mathsf k}(R_1)= \binom{d+1}{2}+ (n-d) -n = n(d-1)-\binom{d}{2}$ 
    minimal generators in degree 2.
    Using the additivity of Hilbert series over the short exact sequence
    $$ 0 \to \Omega_2^R(L) \to R(-1)\oplus R(-2)^{ n(d-1)-\binom{d}{2}} \to \langle x \rangle + \m^2 \to 0,$$
    we obtain 
    \begin{align*}
        H_{\Omega^R_2(L)}(z)&= \left( z+\left(n(d-1)-\binom{d}{2}\right)z^2 \right) H_R(z) - \left( H_R(z) - \left(1+(n-1)z\right)\right).
    \end{align*}
Since $\Omega_2(L)(3)$ is generated in degree zero and has a linear resolution, we have 
\begin{align*}
    \mathcal P_{\Omega_2^R(L)(3)}(z)&= \dfrac{H_{\Omega_2^R(L)(3)}(-z)}{H_R(-z)}\\
    &= \dfrac{\dfrac{1}{(-z)^3} \left( z+\left(n(d-1)-\binom{d}{2}\right)z^2 \right) H_R(z) - \left( H_R(z) - \left(1+(n-1)z\right)\right)}{H_R(-z)}\\
    &= \dfrac{1}{z^3} \left\{(1+z- \left(n(d-1) - \binom{d}{2} \right)z^2 - \dfrac{1-(n-1)z}{H_R(-z)} \right\}.
\end{align*}
Computing the coefficient of $z^{d+2}$, we get \begin{equation*}
    \beta_{d+1,d+2}^R(L) = (d-1) \sum\limits_{i=1}^d \binom{d}{i} (n-d)^{d+1-i}.
\end{equation*}
Note that since $R$ is not a polynomial ring, we have $n>d$. In particular, the Betti number in the above equation is nonzero.

Now, observe that by Theorem~\ref{thm:purityImpliesKoszul},  $R$ is Koszul. Hence,  $\m^2(2)$ has a linear linear resolution, and it is generated in degree zero. Thus, computing $\beta_{d+1,d+2}^R(R/\m^2)$ in a similar manner as above, we get
\begin{equation*}
    \beta_{d+1,d+2}^R(R/\m^2)= (d )\  \sum\limits_{i=1}^d \binom{d}{i} (n-d)^{d+1-i}.
\end{equation*}
Since $\beta_{0,0}^R(L) \neq 0$ and $\beta_{i,i+i}^R(L) \neq 0$ for all $i \geq 1$, the hypothesis on the extremal rays implies that in every decomposition of $\beta^R(L)$ into the extremal Betti tables, $\beta^R(R/\m^2)$ appears with coefficient $$\beta^R_{d+1, d+2}(L)/ \beta^R_{d+1, d+2}(R/\m^2) =  (d-1) /d.$$ On the other hand, we observe that 
\begin{align*}
    \beta_{1,2}^R(L) - \dfrac{d-1}{d} \beta_{1,2}^R(R/\m^2) &= \left( n(d-1) - \binom{d}{2} \right) - \dfrac{d-1}{d} \left( \binom{d+1}{2} + (n-d) d \right) \\
    &= \dfrac{1-d}{d} < 0,
\end{align*}
which is a contradiction, since every entry in $\beta^R(L)-\dfrac{d-1}{d}\beta^R(R/\m^2)$ has to be nonnegative. This shows that $d=1$.

Finally, since $\dim(R)=1$, the last part of the statement follows from Theorem~\ref{dim0and1extremals}(b)(ii).
\end{proof}

A standard graded algebra $R$ is said to have \emph{minimal multiplicity} if $e(R)=\edim(R)-\dim(R)+1$ (see \cite{EG84}). Thus, from Theorem~\ref{thm:dim-at-most-1-for-smallest-set-of-extremals} and Theorem~\ref{dim0and1extremals}, the following characterization is immediate. 

\begin{cor}
    Let $R$ be a standard graded $\sk$-algebra. Then the following are equivalent. 
    \begin{enumerate}[{\rm (i)}]
        \item $R$ is a polynomial ring or Cohen--Macaulay ring of dimension at most one with minimal multiplicity.
        \item The extremal rays of $\mathbb B_{\mathbb Q}(R)$ are spanned by Betti tables of shifts of the following modules:
\begin{enumerate}[{\rm (a)}]
            \item $R/\m^i$, where $i\geq 0$
            \item Pure modules $M$ of finite projective dimension with $\codim(M)=\pdim(M)$.
\end{enumerate} 
    \end{enumerate}
\end{cor}

\section{Rings with Hilbert Series \texorpdfstring{$\dfrac{1+nz}{(1-z)^d}$}{(1+nz)/(1-z)^d}}
\label{special-Hilbert-series}

This section focuses on understanding the graded Betti numbers of modules over standard graded algebras having minimal multiplicity. We also give a complete description of the Betti cone of maximal Cohen--Macaulay modules over these algebras. 

\begin{rmk}\label{rmk:min-mult-is-Koszul}{\rm
Let $R$ be a Cohen--Macaulay standard graded $\mathsf k$-algebra with $H_R(z)= (1+nz)/(1-z)^d$, and $\underline{x}$ be a maximal regular sequence of linear forms in $R$. Let $\overline{R}= R/ \langle \underline{x}\rangle$. Then $H_{\overline{R}}(z)=1+nz$, and hence, $\Omega_1^{\overline{R}}(\sk) \simeq \overline{\m} \simeq \sk(-1)^n$. Therefore, $\sk$ has linear resolution over $\overline{R}$, i.e., $\overline{R}$ is Koszul, and $\beta^{\overline{R}}_i(\sk)=n^i$.
Since $\overline{R}$ is Koszul, by Remark \ref{Rmk-transfer-of-Koszulness}, it follows that $R$ is Koszul.
}\end{rmk}

\begin{prop}\label{Prop-betti-nos-of-R/mi}
Let $R$ be a Cohen--Macaulay standard graded $\sk$-algebra with $H_R(z)= (1+nz)/(1-z)^d$. Then given any $i\geq 1$, we have $\beta_{j+1}(R/\m^i) = n \beta_j(R/\m^i)$ for all $j \geq d$.
\end{prop}
\begin{proof}
Let $f(z)=H_{R/\mathfrak{m}^i}(z)$. Then $\deg(f(z))= i-1$. Now,
\[ H_{\mathfrak{m}^i}(z)= H_R(z) - H_{R/\mathfrak{m}^i}(z) = H_R(z)-f(z).\]
Since $R$ is Koszul, $\m$ has linear resoution. Hence, by Theorem \ref{Thm-reg-over-poly-ring}, $\m^i$ has 
 linear resolution.  Therefore,  we have $H_{\mathfrak{m}^i(i)}(z) = H_R(z) \mathcal{P}_{\mathfrak{m}^i(i)}(-z)$. Since $H_{\mathfrak{m}^i}(z)= z^i H_{\mathfrak{m}^i(i)}(z)$, we get
\begin{equation*}
\begin{split}
 \mathcal{P}_{\mathfrak{m}^i(i)}(-z) &= H_{\mathfrak{m}^i}(z) / (z^iH_R(z))\\
                                                 &= (H_R(z)-f(z))/( z^i H_R(z))\\
                                                 &= (1+nz-(1-z)^d f(z))/z^i (1+nz).
\end{split}
\end{equation*}  

 Note that from $\deg(f(z))=\deg(f(-z))=i-1$ we get
\[\mathcal{P}_{\mathfrak{m}^i(i)}(z) =(1/z^i)(1+nz+n^2z^2+\cdots)(a_0+a_1z+\cdots+a_{d+i-1}z^{d+i-1})\]
for some $a_j \in \BZ$. Now, from the following Remark \ref{rmk:series}, it follows that  for every $j \geq d-1$ the coefficient of $z^{j+1}$ on the right hand side of the above equation is $n$ times the coefficient of $z^{j}$, i.e., $\beta_{j+1}(R/\m^i)=n \beta_j(R/\m^i)$ for all $j \geq d$.
%
\end{proof}

\begin{rmk}\label{rmk:series}{\rm 
Let $f(z) = \sum_{j=\ell}^{m} a_j z^j$ for some $\ell \leq m \in \mathbb{Z}$, and $f(z)/(1-nz) = \sum_j b_j z^j$. Then for all $j \geq m$ we have $b_{j+1}=\sum_{k=\ell}^{m} a_k n^{j+1-k}= n \sum_{k=\ell}^{m} a_k n^{j-k} = n b_j$.
}\end{rmk}

The following result describes the Betti cone of maximal Cohen--Macaulay modules over Cohen--Macaulay algebras of minimal multiplicity.
\begin{prop}\label{MCM_Betti_Cone}
 Let $R$ be a Cohen--Macaulay standard graded $\sk$-algebra with $H_R(z)=\dfrac{1+nz}{(1-z)^d}$. Suppose that $\underline{x}$ is a maximal $R$-regular sequence of linear forms, and $\overline{R}= R/ \langle \underline{x} \rangle$. Let $M$ be a maximal Cohen--Macaulay $R$-module. Then 
\begin{enumerate}[(a)]
 \item $n\beta_{0,j}(M)\geq\beta_{1,j+1}(M)$, and $n \beta_{i,j}(M)=\beta_{i+1,j+1}(M)$ for all $i \geq 1$.
 
  In particular, if every minimal generator of $M$ is in degree $j$, then $M$ is linear, and $\beta^{\overline{R}}(\sk(-(j+1)))= c \beta^R(\Omega_1(M))$ for some $c \in \mathbb Q_{>0}$. 

 \item $\mathbb{B}_{\mathbb{Q}}^{\text{MCM}}(R)= \mathbb{B}_{\mathbb{Q}}(\overline{{R}})$.
 
\item The extremal rays of $\BB_{\mathbb{Q}}^{\text{MCM}}(R)$ are spanned by the Betti tables of shifts of $R$ and $\Omega_d(\sk)$. 
\end{enumerate} 
 \end{prop}
 
 \begin{proof}
 
 Since $M$ is maximal Cohen--Macaulay, $\underline{x}$ is $M$-regular. Hence, $\beta^R(M)=\beta^{\overline{R}}(M/\langle\underline{x}\rangle M)$. In particular, this proves, $\mathbb{B}_{\mathbb{Q}}^{\text{MCM}}(R)\subset\mathbb{B}_{\mathbb{Q}}(\overline{{R}})$.

(a)  Using Theorem \ref{dim0and1extremals}, we have
$$ \beta^{R}(M) =\beta^{\overline{R}}(M/ \langle\underline{x}\rangle M)= \sum_{j\in \mathbb{Z}} c_{1j} \beta^{\overline{R}}(\overline{R}(-j)) + \sum_{j \in \mathbb{Z}}c_{2j} \beta^{\overline{R}}(\sk(-j))$$
for some $c_{1j}, c_{2j} \in \mathbb{Q}_{\geq 0}$. 
Since $H_{\overline{R}}(z)=1+nz$, $\sk$ has linear resolution over $\overline{R}$ with $\beta^{\overline{R}}_i(\sk)=n^i$. This proves (a) since $\beta_{i,j}^{\overline{R}}(\overline{R})= 0 $ for $i \geq 1$.

(b) We have already seen that $\mathbb{B}_{\mathbb{Q}}^{\text{MCM}}(R)\subset\mathbb{B}_{\mathbb{Q}}(\overline{{R}})$. 
To prove the other inclusion, we first note that since $H_{\overline{R}}(z)= 1+nz$, by Theorem \ref{dim0and1extremals}, the extremal rays of ${\overline{R}}$ are spanned by the Betti tables of shifts of $\sk$ and $\overline{R}$. Therefore, it suffices to show that there exists a maximal Cohen--Macaulay $R$-module $M$ such that $\beta ^{\overline{R}}(\sk) = c\beta^R(M)$  for some $c \in \mathbb{Q}_{>0}$. 

Since $\Omega_d^R(\sk)$ is a maximal Cohen--Macaulay $R$-module generated in degree $d$ with linear resolution, by Proposition \ref{Prop-betti-nos-of-R/mi} and (a) above, we get $\beta ^{\overline{R}}(\sk) = (1/ \beta_{d}(\sk))\beta^R(\Omega_{d}^R(\sk)(d))$. 

(c) Follows from the proof of (b) above.
\end{proof}

Recall that the \emph{curvature} of an $R$-module $M$ is defined as $\curv(M) =\limsup\limits_{j} \sqrt[j]{\beta_{j}^R(M)}$. It is known (e.g. see \cite[Proposition 4.2.4]{Av98}) that $\curv(\sk)\geq \curv(M)$ for all finitely generated $R$-modules. The modules $M$ for which the equality $\curv(M)=\curv(\sk)$ holds are said to have \emph{maximal curvature}.  As a consequence of Proposition~\ref{MCM_Betti_Cone} above, we obtain the following interesting result, which was also noted in \cite[Example 5.2.8]{Av98}.

\begin{cor}
    Let $R$ be a Cohen--Macaulay standard graded $\sk$-algebra with $H_R(z)=\dfrac{1+nz}{(1-z)^d}$. Then every $R$-module of infinite projective dimension has maximal curvature.
\end{cor}
\begin{proof}
    By Proposition \ref{MCM_Betti_Cone}, we see that given any module $M$ of infinite projective dimension, we have $\beta_{j+1}(M)=n \beta_{j}(M)$ for $j \gg 0$. Hence, $\curv(M)=\limsup\limits_{j} \sqrt[j]{\beta_{j}^R(M)} = n$. 
\end{proof}

\begin{rmk}{\rm
If $R$ is Cohen--Macaulay with $H_R(z)=\dfrac{1+nz}{(1-z)^d}$, and if $M$ is an $R$-module with $\depth(M)=d-c$, then $\Omega_{c}(M)$ is maximal Cohen--Macaulay.

\begin{enumerate}[(a)]
\item  By Proposition \ref{MCM_Betti_Cone}, it follows that $n \beta_{i,j}(\Omega_{c}(M))=\beta_{i+1,j+1}(\Omega_{c}(M))$ for all $i \geq 1$, and $n\beta_{0,j}(\Omega_{c}(M))\geq\beta_{1,j+1}(\Omega_{c}(M))$.

\item If $M$ is pure of type $(\delta_0, \delta_1,\delta_2, \ldots)$, then $\Omega_{c}(M)$ has a linear resolution. Moreover, by (a) above, it follows that $\mathcal{P}_{\Omega_{c+1}^R(M)}(z) = \dfrac{\beta_{c+1}^R(M) z^{\delta_{c+1}}}{1-nz}$.
\end{enumerate}
}\end{rmk}

Proposition~\ref{MCM_Betti_Cone} also leads to the following question. 
\begin{question}\label{Que:BettiRelAtCodim} Let $R$ be a standard graded $\sk$-algebra with $H_R(z)=(1+nz)/(1-z)^d$, and $M$ be a pure Cohen--Macaulay module of infinite projective dimension with $\codim(M)=c$,  such that $\beta^R(M)$ spans an extremal ray in $\BB_{\mathbb{Q}}(R)$. Then is $\beta_{c+1}(M)= n \beta_{c}(M)$?
\end{question}

Note that if $R$ is Cohen--Macaulay of dimension $d$, then by \cite[Theorem 3.9]{AK18} or \cite{IMW23}, we know that for every degree sequence ${{\bf d}}= (\delta_0,\ldots,\delta_p,\infty, \infty,\ldots)$ with $p \leq d$, there exists a pure Cohen--Macaulay module of type ${\bf d}$. Moreover, by Theorem \ref{SomeExtremalRays}, for such a module $M$, $\beta(M)$ is extremal in $\BB_{\mathbb{Q}}(R)$. This discussion, along with Proposition \ref{MCM_Betti_Cone}, gives us the following.

\begin{prop} Let $R$ be a Cohen--Macaulay standard graded $\sk$-algebra with $H_R(z)=\dfrac{1+nz}{(1-z)^d}$.  If $M$ is a pure $R$-module, then the degree sequence of $M$ has the form $(\delta_0,\ldots,\delta_p,\infty,\infty,\ldots)$ for some $p\leq d$, or $(\delta_0,\ldots,\delta_d, \delta_{d}+1, \delta_d+2, \delta_d+3, \ldots)$. 

Moreover, for every degree sequence ${\bf d}=(\delta_0,\ldots,\delta_p,\infty,\infty,\ldots)$ with $p\leq d$ there exists a pure Cohen--Macaulay $R$-module of type ${\bf{d}}$.
\end{prop}

\begin{prop}\label{extremalImpliesbeta1=nbeta0}
 Let $R$ be a Cohen--Macaulay  standard graded $\sk$-algebra with $H_R(z)=\dfrac{1+nz}{(1-z)^d}$, where $n>0$. Let $M$ be a maximal Cohen--Macaulay module generated in a single degree. If $\beta^R(M)$ spans an extremal ray of $\mathbb{B}_{\mathbb{Q}}(R)$, then either $M$ is free or $\beta^R(M)= c \beta^R(\Omega)$  for some shift $\Omega$ of $\Omega_d^R(\sk)$ and $c \in \BQ_{>0}$.
 
\end{prop}
\begin{proof}
We may assume that $M$ is non-free, and generated in degree zero.
Let $\Omega$ be the shift of $\Omega_{d}(\sk)$ generated in degree zero. Note that by Proposition \ref{Prop-betti-nos-of-R/mi}, we have $\beta_1(\Omega)=n \beta_0(\Omega)$.

Using Proposition \ref{MCM_Betti_Cone}(c), we get that 
$$  \beta^R(M)= \dfrac{\beta_1(M)}{n \beta_{d}(\sk)}\beta^R(\Omega)+ \left(\beta_0(M)-\dfrac{\beta_1(M)}{n}\right)\beta^R(R).$$

Since $\beta^R(M)$ is extremal, and $n, \beta_1(M)$ and $\beta_{d}(\sk)$ are all nonzero, we get $n\beta_0(M)-\beta_1(M)=0$, i.e., 
$ \beta^R(M) = (\beta_1(M)/n \beta_{d}(\sk))\beta^R(\Omega)$.
\end{proof}

\begin{rmk}{\rm
The converse of the above proposition is false. For instance, consider the ring $R= \sk[X,Y]/\langle XY \rangle$, and $M=\langle X \rangle$. Then $H_R(z)= (1+z)/(1-z)$, and $R$ is Cohen--Macaulay. Also, $M$ is a maximal Cohen--Macaulay module $R$-module, and we have $\beta^R(M) = \dfrac{1}{2}\beta^R(\Omega_1(\sk))$. 
However, $\beta^R(M)$ does not span an extremal ray of $\mathbb{B}_{\mathbb{Q}}(R)$ since 
$$\beta^R(M) = \dfrac{1}{2}\beta^R(R) + \dfrac{1}{2}\beta^R(\sk).$$

However, by Proposition \ref{MCM_Betti_Cone}(c), we see that the converse holds if we consider the cone $\mathbb{B}^{\text{MCM}}_{\mathbb{Q}}(R)$ spanned by the Betti tables of maximal Cohen--Macaulay $R$-modules.
}\end{rmk}

As observed in Remark \ref{rmk:min-mult-is-Koszul}, the Koszul property is a necessary condition for a Cohen--Macaulay algebra to have minimal multiplicity. Recall that Proposition \ref{MCM_Betti_Cone}(c) states that if $R$ is Cohen--Macaulay of minimal multiplicity, then the extremal rays of $\mathbb{B}^{\mathrm{MCM}}_{\mathbb{Q}}(R)$ are spanned by the Betti tables of shifts of $R$ and $\Omega_d^R(\mathsf{k})$. We end this section by showing that this property indeed characterizes Cohen--Macaulay algebras of minimal multiplicity.
\begin{thm}\label{thm:min-mult-char}
    Let $R$ be a standard graded Cohen--Macaulay  Koszul $\sk$-algebra of dimension $d$. Then the following are equivalent: 
    \begin{enumerate}
        \item[{\rm (i)}] $R$ has minimal multiplicity.
        \item[{\rm (ii)}] The extremal rays of $\mathbb{B}^{\mathrm{MCM}}_{\mathbb{Q}}(R)$ are spanned by the Betti tables of shifts of $R$ and $\Omega_d^R(\mathsf{k})$.
    \end{enumerate}
\end{thm}
\begin{proof}
The implication (i) $\Longrightarrow$ (ii) is the content of Proposition~\ref{MCM_Betti_Cone}(c). We now prove (ii) $\Longrightarrow$ (i). So, assume that (ii) holds.

Let $\mathbb F_\bullet$ denote the minimal graded $R$-free resolution of $k$, with $\beta_i=\beta_i^R(\mathsf k)$. Since $R$ is Koszul, $F_i=R(-i)^{\beta_i}$ for every $i\geq 0$.  From the exact sequence $
0\to\Omega^R_d(\mathsf k)\to F_{d-1}\to\cdots\to F_0\to \mathsf k\to 0$, we obtain the following relation between Hilbert series:
\[
H_{\Omega^R_d(\mathsf k)}(z)=\left(\sum_{i=0}^{d-1}(-1)^{d-1-i}\beta_i\, z^i\right) H_R(z)+(-1)^d.
\]
For any maximal Cohen--Macaulay module $M$, let $h_M(z)$ denote the polynomial $ (1-z)^d H_M(z)$.  Then, the above equation gives
\begin{equation}
h_{\Omega^R_d(\mathsf k)}(z)=\left(\sum_{i=0}^{d-1}(-1)^{d-1-i}\beta_i\, z^i\right) h_R(z)+(-1)^d(1-z)^d.
\end{equation}

Since $\Omega_{d+1}^R(\sk)$ is maximal Cohen--Macaulay and generated in degree $d+1$, by hypothesis (ii), there exist $c_1, c_2 \in \mathbb Q_{\geq 0}$ such that
\begin{equation*}
\beta^R\bigl({\Omega_{d+1}^R(\mathsf k)}\bigr) = c_1\,\beta^R\bigl(R(-(d+1))\bigr) + c_2\,\beta^R\bigl(\Omega_d^R(\mathsf k)(-1)\bigr)
\end{equation*}
for some $c_1,c_2\in\mathbb Q_{\ge0}$. 
Hence, we have the relation 
$H_{\Omega^R_{d+1}(\mathsf k)}(z)=c_1\, z^{d+1}\, H_R(z)+c_2\, z\, H_{\Omega^R_d(\mathsf k)}(z)$.
Also, from $0\to\Omega^R_{d+1}(\mathsf k)\to F_d\to\Omega^R_d(\mathsf k)\to 0$, we get the relation
$H_{\Omega^R_{d+1}(\mathsf k)}(z)=\beta_d\, z^d\, H_R(z) - H_{\Omega^R_d(\mathsf k)}(z)$.
Equating the two expressions of $H_{\Omega_{d+1}^R(\mathsf k)}(z)$, and multiplying by $(1-z)^d$, we obtain  $$\bigl(\beta_d z^d - c_1 z^{d+1}\bigr)h_R(z) = (1+c_2 z)\,h_{\Omega^R_d(\mathsf k)}(z).$$
Substituting for $h_{\Omega^R_d(\mathsf k)}(z)$ using equation (1), we get
\[
\bigl(\beta_d z^d - c_1 z^{d+1}\bigr)h_R(z)
= (1+c_2 z)\left[\left(\sum_{i=0}^{d-1}(-1)^{d-1-i}\beta_i\, z^i\right) h_R(z)+(-1)^d(1-z)^d\right],
\]
which rearranges to
\begin{equation*}
h_R(z)\left[\beta_d z^d - c_1 z^{d+1} - (1+c_2 z)\left(\sum_{i=0}^{d-1}(-1)^{d-1-i}\beta_i z^i\right)\right]
= (-1)^d(1+c_2 z)(1-z)^d.
\end{equation*}
Note that the right-hand side of the above equation is nonzero. Furthermore, $h_R(1)=e(R)\neq 0$. This forces $h_R(z)\mid (1+c_2z)$, and hence $\deg\bigl(h_R(z)\bigr)\le 1$. Let $h_R(z)=1+nz$. Then we have $H_R(z)=(1+nz)/(1-z)^d$. Since $R$ is standard graded, it follows that $n=\edim(R)-d$. This shows that $R$ has minimal multiplicity.
\end{proof}

\section{Concluding Remarks}\label{sec:concluding-remarks}

The main question we have addressed in this article is the following.

\begin{question}\label{que-in-concluding}
    Let $R$ be a standard graded $\mathsf{k}$-algebra. When does the equality 
    \[
    \mathbb{B}_{\mathbb{Q}}(R) = \mathbb{B}_{\mathbb{Q}}^{\mathrm{pure}}(R)
    \]
    hold?
\end{question}

Suppose that $\edim(R) \geq 2$ and that the equality $\mathbb{B}_{\mathbb{Q}}(R) = \mathbb{B}_{\mathbb{Q}}^{\mathrm{pure}}(R)$ holds. Then we have shown that:
\begin{enumerate}[{\rm (i)}]
    \item $R$ is a Koszul algebra;
    \item $R$ is \emph{good} in the sense of Roos;
    \item if $\depth(R) = 0$, then $R$ is Artinian and level.
\end{enumerate}

From \cite{AK20}, it is known that if $\mathbb{B}_{\mathbb{Q}}(R) = \mathbb{B}_{\mathbb{Q}}^{\mathrm{pure}}(R)$, then the Betti tables of shifts of the following modules are extremal:
\begin{enumerate}[{\rm (a)}]
    \item $R/\m^j$, for $j \geq 0$;
    \item modules $M$ with a pure resolution satisfying $\codim(M) = \pdim_R(M)$.
\end{enumerate}
Our results characterize all standard graded $\mathsf{k}$-algebras $R$ for which these modules form a complete set of extremal rays of $\mathbb{B}_{\mathbb{Q}}(R)$: they are either polynomial rings or Cohen--Macaulay algebras of dimension at most one with minimal multiplicity. We have also obtained a characterization of Cohen--Macaulay algebras of minimal multiplicity in terms of the extremal rays of the Betti cone of its maximal Cohen--Macaulay modules.

Gorenstein Artin algebras of embedding dimension $2$ with $\m^3 = 0$ are known to satisfy 
$
\mathbb{B}_{\mathbb{Q}}(R) = \mathbb{B}_{\mathbb{Q}}^{\mathrm{pure}}(R).
$
We have shown that this phenomenon does not extend to Gorenstein algebras of higher socle degree: generic Gorenstein Artin algebras of socle degree at least $3$ fail to satisfy the equality.

Despite the necessary conditions and results obtained, a complete answer to Question~\ref{que-in-concluding} remains elusive. We propose the following conjecture:

\begin{conjecture}
   Let $R$ be a standard graded $\sk$-algebra such that $\BB_{\mathbb{Q}}(R)=\BB^{\rm{pure}}_{\mathbb{Q}}(R)$. Then $R$ is a Cohen--Macaulay level algebra.
\end{conjecture}


\providecommand{\bysame}{\leavevmode\hbox to3em{\hrulefill}\thinspace}
\providecommand{\MR}{\relax\ifhmode\unskip\space\fi MR }
\providecommand{\MRhref}[2]{%
  \href{http://www.ams.org/mathscinet-getitem?mr=#1}{#2}
}
\providecommand{\href}[2]{#2}

\end{document}